\documentclass[12pt,oneside,a4paper,final]{article}
\usepackage[left=2cm,right=3cm,top=1cm,bottom=2.5cm,includehead]{geometry} % Seitenränder
\usepackage[utf8]{inputenc}
\usepackage[english]{babel}
\usepackage{amsmath}
\usepackage{amssymb}
\usepackage{amsthm}
\usepackage[titletoc,page]{appendix}
\usepackage{xifthen}
\usepackage{enumitem}
\usepackage{xcolor}
\usepackage{csquotes}
\usepackage{ifdraft}
\usepackage[authoryear]{natbib}
\usepackage{tikz}
\usepackage{float}
\usepackage[font=footnotesize]{caption}
\newtheorem{thm}{Theorem}[section]

\newtheorem{lemma}[thm]{Lemma}

\newtheorem{cor}[thm]{Corollary}

\newtheorem{rem}[thm]{Remark}

\usepackage[pagebackref]{hyperref}
\hypersetup{
    final,
    colorlinks=true,
    raiselinks=true, % Zeilenumbrüche in links
    linkcolor=[rgb]{.0, .0, .8}, % Interne Links
    citecolor=[rgb]{.0, .0, .8}, % Quellenverweise
    filecolor=black, % Dateilinks
    urlcolor=black, % Adressen im web
}
\usepackage{chngcntr}
\counterwithin{equation}{section}

\renewenvironment{proof}[1][]{{\bfseries \itshape Proof\ifthenelse{\isempty{#1}}{}{ #1}: }}{\nopagebreak[4]\par\qed\par\br}

\usepackage{fancyhdr}
\fancypagestyle{plain}{

    \fancyhead{}
    \fancyfoot[C]{\thepage}
}
\newcommand{\var}[1]{\operatorname{Var}\left(#1\right)}
\newcommand{\cov}[2]{\operatorname{cov}\left(#1, #2\right)}
\newcommand{\e}[1]{\mathbb{E} \left[#1\right]}
\newcommand{\ex}{\mathbb{E}}
\newcommand{\p}[1]{\mathbb{P} \left(#1\right)}
\newcommand{\pr}{\mathbb{P}}
\newcommand{\inn}[1]{\operatorname{Int}\left(#1\right)}
\newcommand{\br}{\bigskip}

\newcommand{\bea}{\begin{eqnarray*}}
\newcommand{\eea}{\end{eqnarray*}}
\newcommand{\be}{\begin{eqnarray}}
\newcommand{\ee}{\end{eqnarray}}
\newcommand{\beq}{\begin{equation}}
\newcommand{\eeq}{\end{equation}}

\def\sp{\mathcal{S}_{p}}
\def\R{\mathbb{R}}
\def\N{\mathbb{N}}

\begin{document}
    %~ \nocite{*} % Don't hide unused references

    \title{Determinants of Random Block Hankel Matrices}

    \author{
        {\small Dominik Tomecki}\\
        {\small Ruhr-Universit\"at Bochum }\\
        {\small Fakult\"at f\"ur Mathematik}\\
        {\small 44780 Bochum, Germany}\\
        {\small email: dominik.tomecki@rub.de}\\
        \and
        {\small Holger Dette} \\
        {\small Ruhr-Universit\"at Bochum} \\
        {\small Fakult\"at f\"ur Mathematik} \\
        {\small 44780 Bochum, Germany} \\
        {\small email: holger.dette@rub.de}
    }

    \maketitle

    \begin{abstract}
    
    We consider the moment space $\mathcal{M}_{2n+1}^{p_n} $ of moments up to the order $2n+1$ of  $p_n \times p_n$ real matrix measures defined on the interval $[0,1]$. 
    The asymptotic properties of the Hankel determinant $$\big \{  \log  \det (M_{i+j} ^{p_n} ) _{i,j=0,\ldots , \lfloor nt  \rfloor}  \big   \} _{ t \in [0,1]} 
    $$ 
of a uniformly distributed vector $(M_{1}, \dots, M_{2n+1})^t \sim \mathcal{U} (\mathcal{M}_{2n+1})$ are studied when the dimension $n$ of the moment space and the size of the matrices $p_n$ converge to infinity. In particular weak convergence of an appropriately centered and standardized version of
this process is established. Mod-Gaussian convergence is shown and several large and moderate deviation principles are derived.
Our results are based on  some new relations between determinants of subblocks of the Jacobi-beta-ensemble, which are of their own
interest and generalize  Bartlett decomposition-type results for the Jacobi-beta-ensemble from the literature.

    \end{abstract}



        %~ % Table of contents
    %~ {
        %~ \hypersetup{linkcolor=black}
        %~ \makeatletter
        %~ \renewcommand*\l@section{\@dottedtocline{1}{1.5em}{2.3em}}
        %~ \makeatother
        %~ \tableofcontents
    %~ }

    % Introduction
\section{Introduction}

A  $p \times p$  matrix $\mu = (\mu_{i,j})^p_{i,j=1}$ of signed real  measures $\mu_{i,j}$, such that for each Borel set $A \subset [0,1]$ the matrix $\mu(A) = (\mu_{i,j}(A))^p_{i,j=1}$ is symmetric and nonnegative definite, is called matrix measure  on the interval $[0,1]$.
Matrix measures have been studied extensively in the literature generalising many classical results  in the context of  moment theory, orthogonal
polynomials, quadrature formulas [see  \cite{krein1949} for an early reference and   \cite{durass1995},
 \cite{duran1999} and \cite{durlop1997}, \cite{gruenbaum2003}, \cite{grupactir2005}, 
\cite{dampussim2008}, \cite{gamnag2012} and \cite{gamnag2016}  for some more recent references among many others].  \\
In a recent paper \cite{detnag2012a}  defined  $ \mathcal{P}_p$ as the set of all matrix measures on the interval $[0,1]$ 
satisfying the condition  $ \int^1_0 d\mu(x) =I_p $ [here and throughout this paper $I_p$ denotes the $p \times p$  identity matrix]
and  studied a uniform distribution on the  moment space 
\be \label{1.1a}
\mathcal{M}_n^p= \Big \{ (M_1,\dots,M_n)^t \  \Big | \ M_j = \int^1_0 x^j d\mu (x), \ j=1,\dots,n ~,~\mu  \in \mathcal{P}_p \Big \}
\ee
of all (matrix) moments 
\be \label{1.1}
M_k = \int^1_0 x^k d\mu(x) =  \Big(\int_0^1  x^k d \mu_{i, j}(x)\Big )_{1 \le i, j \le p} ~;\quad k=0,1,2,\dots
\ee
up to the order $n$ corresponding to matrix  measures on the interval $[0,1]$. Note that $ \mathcal{M}_n^p  \subset
(\mathcal{S}_p)^n $, where $\mathcal{S}_p$ denotes the set of all nonnegative definite matrices of size $p \times p$.
If $(M_{1,n}^p, \dots, M_{n,n}^p)^t$ is a uniformly distributed vector on $\mathcal{M}_n^p$, 
\cite{detnag2012a}  showed that an appropriately centered and standardized version of the vector $(M_{1,n}^p, \dots, M_{k,n}^p)^t$ converges weakly 
to a  vector of $k$ independent $p \times p$ Gaussian ensembles, thus generalising the meanwhile classical results of \cite{chakemstu1993}
for the case $p=1$.
 \\
The one-dimensional case has also been studied intensively with respect to other properties of random moment sequences and we refer to 
\cite{gamloz2004}, \cite{lozada2005} for large deviation results and to \cite{detnag2012b} for some results on more general moment spaces. 
Recently \cite{dettom2016} examined the asymptotic properties  of a stochastic process of  Hankel determinants 
$$
\big \{  \log  \det (M_{i+j,2n}^p) _{i,j=0,\ldots , \lfloor nt  \rfloor}  \big   \} _{t \in [0,1]}
$$
of a uniformly distributed moment vector $(M_{1,2n}^p, \ldots , M_{2n,2n}^p )^t$ on the moment space  $\mathcal{M}_{2n}^p   $ in the case $p=1$ and derived weak convergence and large deviation principles for this process. 

In the present paper we will investigate properties of  a stochastic process corresponding to the determinant of  matrix valued random Hankel matrices, where
the dimension of the moment space and the dimension $p_n$ of the matrix measures converge to infinity. To be 
precise, consider a uniformly random vector $(M_{1, 2n + 1}^{p_n} , \ldots , M_{2n+1, 2n + 1}^{p_n })^t$ on the moment space $\mathcal{M}_{2n+1}^{p_n}$, where $p_n$ is a sequence of integers converging 
to infinity as $n\to \infty$ and define the stochastic process  
\begin{equation} \label{basicproc}
{H}_{n} (t) = 
\big \{  \log  \det (M_{i+j, 2n + 1} ^{p_n} ) _{i,j=0,\ldots , \lfloor nt  \rfloor}  \big   \} _{ t \in [0,1]}.
\end{equation} 
We establish weak convergence of the process  $ \{{H}_{n} (t)   \} _{t \in [0,1]} $ with a Gaussian limit, Mod-Gaussian convergence (for fixed $ t \in [0,1]$)
and derive several moderate and large  deviation principles. In Section~\ref{section_basics} we will present 
 some basic facts about matrix-valued moment spaces. 
Section~\ref{section_distribution} is devoted to the investigation of  distributional properties of  determinants corresponding to subblocks 
of  the Jacobi-beta-ensemble.  These results are of own interest and provide a new point on classical 
results about a Bartlett- type decomposition [see  \cite{bart1933}]  for the Jacobi-beta-ensemble [see  for example \cite{kshi1961}]. In Section
\ref{section_applied_convergence} we prove weak convergence of the  process \eqref{basicproc}.
 Finally, in Section~\ref{section_mod_phi} we  examine mod-$\phi$-convergence, as well as moderate and large deviations. 
 Our results are based on several delicate estimates  of the cumulants of logarithms of beta-distributed random variables, stated in the 
  Appendix, which also contains proofs of the more technical results and  some  inequalities about polygamma functions.

    % Matrix-valued moment spaces
\section{Moment spaces of  matrix-valued meausres} 
\label{section_basics}

  We begin recalling some basic facts about the moment space 
       $\mathcal{M}_n^p$    defined in \eqref{1.1a}  [see   \cite{detstud2002}  for a detailed discussion].
   In the following we compare matrices   with respect to the Löwner (partial) ordering. Thus for two $p\times p$ symmetric matrices $A, B $ 
  we use the notation $   A < B$  ($ A \le B $) if and only if the difference $B - A$  is positive definite (positive semi-definite). 
We denote by $\mathcal{S}_p$ the set of nonnegative definite (symmetric) $p \times p$ matrices. Let $\mathcal{B}([0,1])$ denote the Borel field on  the interval $[0,1]$.
A map 
$$
\mu = (\mu_{i, j})_{1 \le i, j \le p}: \mathcal{B}([0,1]) \to \mathbb{R}^{p \times p}
$$ 
is called a \emph{matrix-valued measure}, if $\mu_{i, j}$ is a signed measure for all $1 \le i, j \le p$ and $\mu(A) \in \mathcal{S}_p$ for every Borel set $A \subset  [0,1]$.  We denote by $\mathcal{P}_p$ the set of all $p \times p$ matrix-valued 
measures on  the interval $[0,1]$ satisfying  $\mu([0,1]) = I_p$ and consider the $n$th moment space $\mathcal{M}_n^p$ defined in \eqref{1.1a}, 
    which is a subset of $(\mathcal{S}_p) ^n$. Note that in the one dimensional case $\mathcal{P}_1$ is the set of all probability measures on the interval $[0,1]$.

    \cite{detstud2002} introduced new ``coordinates'' for the moment space $\mathcal{M}_n^p$ defining  a one to one
 map from the interior of  $\mathcal{M}_n^p$ onto the product space
$(\mathcal{E}_p)^n$, where $\mathcal{E}_p$
 denotes the ``cube''
\begin{align*}
{\cal E}_p =\{ A \in \sp  ~|~ 0_p < A < I_p \}.
\end{align*}
Here and throughout the paper, $0_p$ denotes the $p \times p$ matrix with all elements equal to zero and $I_p$ denotes the $p \times p$ identity matrix. 

The new coordinates are called canonical moments  [see \cite{detstud2002}], and they are related to the Verblunsky coefficients,
which have been discussed for matrix measures on the unit circle [see \cite{dampussim2008} and \cite{simon2005a,simon2005b}]. They turn out to be extremely 
useful in analyzing the asymptotic properties of the stochastic process defined in \eqref{basicproc}.
The definition of matrix valued
canonical moments 
  relies on the introduction of Block-Hankel-matrices:
    \begin{align*}
        \underline{H}_{2k} &= 
        \begin{pmatrix}
            M_0 & \cdots & M_k \\
            \vdots & \ddots & \vdots \\
            M_k & \cdots & M_{2k}
        \end{pmatrix}
        &\overline{H}_{2k} = 
        \begin{pmatrix}
            M_1 - M_2 & \cdots & M_k - M_{k + 1} \\
            \vdots & \ddots & \vdots \\
            M_k  - M_{k + 1} & \cdots & M_{2k - 1} - M_{2k}
        \end{pmatrix} \\
        \underline{H}_{2k + 1} &= 
        \begin{pmatrix}
            M_1 & \cdots & M_{k + 1} \\
            \vdots & \ddots & \vdots \\
            M_{k + 1} & \cdots & M_{2k + 1}
        \end{pmatrix}
        &\overline{H}_{2k + 1} = 
        \begin{pmatrix}
            M_0 - M_1 & \cdots & M_k - M_{k + 1} \\
            \vdots & \ddots & \vdots \\
            M_k - M_{k + 1} & \cdots & M_{2k} - M_{2k + 1}
        \end{pmatrix}
    \end{align*}
    which - as in the one dimensional case - can be used to characterize elements of the moment space $\mathcal{M}_n^p$. More precisely,  the vector of matrices
     $M = (M_1, \ldots, M_n)^t \in ( \mathcal{S}_p)^n$  satisfies
         \begin{align*}
        M \in \mathcal{M}_p^n & \text{~ if and only if ~}  \underline{H}_k \ge 0_k, \overline{H}_k \ge 0_k  \text{~ for all ~} k \le n \\
        M \in \inn{\mathcal{M}_p^n} & \text{~ if and only if ~}  \underline{H}_k > 0_k, \overline{H}_k > 0_k  \text{~ for all ~} k \le n~,
    \end{align*}
    where $\inn{\mathcal{M}_p^n}$ denotes the interior of the set $ \mathcal{M}_p^n $.
We now introduce the vectors of matrices
    \begin{align*}
        \underline{h}_{2k}^t &= \big (M_{k + 1}, \ldots, M_{2k}\big ) ~,~
        \underline{h}_{2k - 1}^t = \big (M_k, \ldots, M_{2k - 1}) \\
        \overline{h}_{2k}^t &= \big (M_k - M_{k + 1}, \ldots, M_{2k - 1} - M_{2k}\big ) ~,~
        \overline{h}_{2k - 1}^t = \big (M_k - M_{k + 1}, \ldots, M_{2k - 2} - M_{2k - 1}\big )
    \end{align*}
   and define  the ``extremal'' matrices $    M_1^- = 0_p $,  $M_1^+= I_p $ and  
      $  M_2^+ = M_1$ (the phrase  extremal will shortly become clear).     If  $M$ is an element of the interior of the moment space $\mathcal{M}_p^n$, 
      the extremal moments of larger order are defined as
    \begin{align*}
        M_n^- &= \underline{h}_{n - 1}^t \underline{H}_{n - 2}^{-1} \underline{h}_{n - 1} \quad \text{ for } n \ge 2 \\
        M_n^+ &= M_{n - 1} - \overline{h}_{n - 1}^t \overline{H}_{n - 2}^{-1} \overline{h}_{n - 1} \quad \text{ for } n \ge 3 .
    \end{align*}
\cite{detstud2002}  showed that the extremal moments provide a convenient tool to characterise the moment space $\mathcal{M}_p^n$. In particular 
by considering Schur complements  of $\underline{H}_k$ and  $\overline{H}_k$, they showed that
    \begin{align*}
        (M_1, \ldots, M_n)^t \in \inn{\mathcal{M}_p^n} \text{~ if and only if ~}  M_k^- < M_k < M_k^+  \text{~ for all ~} k \le n ~.
    \end{align*}
 This property is then used to define for  a point
    $ (M_1, \ldots, M_n)^t \in \inn{\mathcal{M}_p^n} $ matrix valued canonical moments   as follows
    \begin{align} \label{defU}
        \mathcal{U}_{i} = (M_i^+ - M_i^-)^{-1/2}(M_i - M_i^-)(M_i^+ - M_i^-)^{-1/2} ~,~i=1, \ldots , n. 
    \end{align}
     Given $M_1, \ldots M_{k-1}$, the moment $M_k$ can be calculated from the canonical moment $U_k$ and therefore
 equation \eqref{defU} defines a one to one mapping
\begin{align*}
\varphi_p   : \left \{ \begin{array}{lll}
& \inn{ \mathcal{M}_n^p } \ & \to {\cal E}_p^n  \\
 &(M_{1,n}^p,\dots,M_{n,n}^p)^t & \mapsto \varphi_p (M_{1,n}^p,\dots,M_{n,n}^p) = (U_{1,n}^p,\dots ,U_{n,n}^p)^t\ ,
\end{array} \right .
\end{align*}
from the interior of the moment space onto ${\cal E}_p^n$.

We conclude this section with  a very  interesting and useful relation between canonical moments and determinants  of Hankel  matrices $\underline{H}_{2n} = \det (M_{i+j,2n}^p)_{i,j=0,\ldots ,n}$ corresponding to a
point $ (M_{1,2n}^p ,  \ldots , M_{2n,2n}^p)^t \in  \mbox{Int}( \mathcal{M}_{2n}^p )$, i.e 
    \begin{align}\label{produkt_h}
        \det \underline{H}_{2n}
        =& \prod \limits_{i = 1}^n \Big(\det(I_p - \mathcal{U}_{2i - 2,2n}^p) \det(\mathcal{U}_{2i - 1,2n}^p) \det(I_p - \mathcal{U}_{2i - 1,2n}^p)  \det(\mathcal{U}_{2i,2n}^p)\Big)^{n - i + 1} \\
        =& \prod \limits_{i = 1}^n  \big \{ \det(\mathcal{U}_{2i - 1,2n}^p)   \det(I_p - \mathcal{U}_{2i - 1,2n}^p)  \det(\mathcal{U}_{2i,2n}^p)
         \big\} ^{n- i + 1}  (\det(I_p - \mathcal{U}_{2i,2n}^p))^{n - i},\nonumber
    \end{align}
    where $U_0 = 0_p$ [see \cite{detstud2005} for a proof].

    % Distribution
\section{The distribution of random Hankel block matrices} \label{section_distribution}
 By identifying a symmetric matrix $M = (m_{i, j})_{i, j = 1}^p \in \mathcal{S}_p$ with the vector containing the entries $m_{i, j}$ for $1 \le i \le j \le p$, we get a  subset of $\mathbb{R}^{p(p + 1)/2}$ with non-empty interior. This identification allows us to integrate on $\mathcal{S}_p$ via the usual Lebesgue-measure
 and to  define distributions on $\mathcal{S}_p$ by specifying their  (Lebesgue-)densities. We are particularly interested in matrix valued Beta distributions
 [see \cite{olkrub1964} or \cite{muir1982}]
  supported on the set $\mathcal{E}_p$ with a density proportional to 
      \begin{align} \label{jacbeta}
  f_{  \gamma , \delta }(U) =     (\det U)^{\gamma - e_1} (\det(I_p - U))^{\delta - e_1} I\{0_p < U < I_p\}.
    \end{align} 
    where  $e_1 = \frac{p + 1}{2}$ and the parameters $\gamma, \delta$ satisfy $\gamma, \delta > e_1 - 1$. 
    These distributions are a special case of the  
    Jacobi-beta-ensemble $J\beta E_p(\gamma, \delta)$, which defines a density to be proportional to  \eqref{jacbeta}, where the constant 
    $e_\beta$ is given by   $e_\beta = 1 + \frac{\beta}{2} (p - 1)$. The parameter $\beta$ varies when the entries of the matrix are are real ($\beta = 1$), complex ($\beta = 2$) or quaternions ($\beta = 4$) [see for example \cite{arashi2011}]. In the present paper, we will only consider the real case.
     
  Consider a uniform distribution  on the   $n$-th moment  space $\mathcal{M}_n^p$ defined in \eqref{1.1} and denote by  $\mathcal{U}_{i, n}^p$  
  the $i$-th canonical moment in \eqref{defU}, then it is shown in Theorem 3.5 of  \cite{detnag2012a} that 
  $\mathcal{U}_{i, n}^p , \ldots , \mathcal{U}_{i, n}^p$ are independent and distributed according to a Jacobi-beta-ensemble, that is
        \begin{align}
            \mathcal{U}_{i, n}^p \sim J\beta E_p(e_1 (n - i + 1), e_1 (n - i + 1)) \label{dist_can_mom}.
        \end{align}
>From this result and formula \eqref{produkt_h} it is obvious that the distribution of determinants of random variables 
governed by the  Jacobi-Ensemble will be essential for the following analysis of the process. Our first main result, which is of independent interest, 
provides an important tool to determine the distribution of the process $\{ H_n(t)\}_{t\in [0,1]} $ defined in \eqref{basicproc}.
Throughout this paper we will use the notation $M^{[k]} \in \R^{k \times k}$ for   the upper left $k \times k$ subblock of the matrix $M \in \R^{p\times p}$ and
$\beta (\gamma, \delta)$ denotes a Beta distribution on the interval $[0,1]$.
    
    \begin{thm}[Subblocks of Jacobi-Ensembles] \label{dist_jacobi}
        Assume $p > 1$ and that       
         $U \sim J\beta E_p(\gamma, \delta)$. If we denote by $V = U^{[p - 1]}$ the upper left $(p - 1) \times (p - 1)$ subblock of the random matrix  $U$, then
        \begin{align*}
        \Big (V, \frac{\det U}{\det V}, \frac{\det(I_p - U)}{\det(I_{p - 1} - V)} \Big) \stackrel{d}{=} J\beta E_{p - 1}(\gamma, \delta) \otimes (p_{p - 1, 1}, (1-p_{p - 1, 1}) p_{p - 1, 2}),
        \end{align*}
        where the random variables $p_{i, 1} \sim \beta(\gamma - i/2, \delta)$ and $p_{i, 2} \sim \beta(\delta - i/2, i/2)$ are beta-distributed and independent.
    \end{thm}

    \begin{proof}
    We consider the Cholesky decomposition of the matrix 
      \begin{align}
            U = \begin{pmatrix}V & B \\ B^t & c\end{pmatrix} =     \begin{pmatrix}
                T^t & 0 \\ 
                t^t & \sqrt{t_p}
            \end{pmatrix}   
                \begin{pmatrix}
                T & t \\ 
                0 & \sqrt{t_p}
            \end{pmatrix}
            = 
            \begin{pmatrix}T^t T & T^t t \\ t^t T & t^t t + t_p\end{pmatrix},\label{dist_jacobi_cholesky}
        \end{align}
        where $V= U^{[p - 1]}$ and   $T$ is an upper triangular   $(p - 1) \times (p - 1)$ matrix with strictly positive entries on the diagonal. 
         The matrix $U$ satisfies almost surely the inequality $0_p < U < I_p$. 
         As a matrix is positive definite if and only if its main subblock and the correpsonding Schur complement are positve definite, we obtain 
for the  terms $V$, $B$ and $c$ in \eqref{dist_jacobi_cholesky}:
        \begin{align}
         & 0_{p - 1} < V \text{ and } c - B^t V^{-1} B > 0 \label{dist_jacobi_schur_1} \\
   & V < I_{p - 1} \text{ and } 1 - c - B^t (I_{p - 1} - V)^{-1} B > 0 \label{dist_jacobi_schur_2}
        \end{align}
 This implies  $c < 1$ and we conclude  $t_p = c - t^t t < 1$. Therefore the random variable 
        \begin{align*}
            v = (1 - t_p)^{-1/2}(I_{p - 1} + T(I_{p - 1} - V)^{-1} T^t)^{1/2} t.
        \end{align*}
        is well defined. \\
  We will now determine the joint density of the random variables $V$, $v$ and $t_p$
  (up to a constant).
     For this purpose note that  the equation $c = t^t t + t_p$ 
        yields the Schur complement of $I_{p - 1} - V$ in $I_p - U$ as
        \begin{align*}
            1 - c - B^t (I_{p - 1} - V)^{-1} B   = &1 - t^t t - t_p - t^t T (I_{p - 1} - V)^{-1} T^t t \\
            =& (1 - t_p)(1 - t^t(I_{p - 1} + T (I_{p - 1} - V)^{-1} T^t)t(1 - t_p)^{-1}) \\
             =& (1 - t_p)(1 - v^tv),
        \end{align*}
   which yields combined with the well-known formula for the    determinant  of a   
        Schur-complement
         \begin{align}
         \det (I_p -U) = \det (I_{p-1} -V)    \cdot  (1 - t_p)(1 - v^tv). 
       \label{dist_jacobi_schur_3a}
        \end{align}
      Similarly, observing \eqref{dist_jacobi_cholesky}
        \begin{align} \label{dist_jacobi_schur_3b}
            \det(V) (c - B^t V^{-1} B) = \det U = (\det(T) \sqrt{t_p})^2 = \det(T^t T) t_p = \det(V) t_p,
        \end{align} 
         which  also  implies
        \begin{align}
            c - B^t V^{-1} B = t_p. \label{dist_jacobi_schur_4}
        \end{align}
        Using (\ref{dist_jacobi_schur_1}), (\ref{dist_jacobi_schur_2}), (\ref{dist_jacobi_schur_3a})  and 
         (\ref{dist_jacobi_schur_3b}) 
      it follows that the density of  $V$, $v$ and $t_p$ is proportional to  the function
        \begin{align}
     g_{  \gamma ,  \delta}(V)         &
            =(\det V)^{\gamma - e_1} t_p^{\gamma - e_1} (\det(I_{p - 1} - V))^{\delta - e_1} ((1 - t_p)(1 - v^tv))^{\delta - e_1}\label{dist_jacobi_density_u} 
            \det D^{-1}\\
            &~~~~~~\times  I\{0_{p - 1} < V < I_{p - 1}\} I\{0 < t_p < 1\} I\{v^tv < 1\} \nonumber,
        \end{align}
       where  $\det D^{-1}$ is  the Jacobi-determinant of the corresponding transformation from $U$ to $V$, $v$ and $t_p$.
       As this transformation  leaves the matrix
       $V$  unchanged  we obtain $\det D^{-1}= \det M^{-1}$, where the matrix $M$ is given by  
        \begin{align*}
            M = \begin{pmatrix}
            \frac{\partial v}{\partial B} & \frac{\partial v}{\partial c} \\[.3cm]
            \frac{\partial t_p}{\partial B} & \frac{\partial t_p}{\partial c}
        \end{pmatrix}
        \end{align*}
and 
        \begin{align*}
            \frac{\partial t_p}{\partial c} &= 1 ~,~
            \frac{\partial t_p}{\partial B} = -2B^t V^{-1}.
        \end{align*}
In order to calculate the remaining elements of the matrix $M$, we simplify the representation of $v$  using the formula
        \begin{align*}
            I_{p - 1} + T(I_{p - 1} - V)^{-1} T^t  = I_{p - 1} + (V^{-1} - I_{p - 1})^{-1} 
            % &= I + T(I_{p - 1} - T^t T)^{-1}T^t\\  &= (I_{p - 1} + (V^{-1} - I_{p - 1})^{-1}) 
            = (I_{p - 1} - V)^{-1},
        \end{align*}
        where the second equality stems from an application of the principal axis transform. 
   From this equation, (\ref{dist_jacobi_cholesky}) and (\ref{dist_jacobi_schur_4}) we can rewrite the vector $v$ as
        \begin{align*}
            v = (1 - c + B^t V^{-1} B)^{-1/2}(I_{p - 1} - V)^{-1/2} (T^t)^{-1} B.
        \end{align*}
     Standard calculus,   observing this representation and (\ref{dist_jacobi_schur_4})
    now   gives 
        \begin{align*}
            \frac{\partial v}{\partial c} &= \frac{1}{2(1 - t_p)} v \\
 %           \frac{\partial v_i}{\partial B_j} &= - \frac{1}{2(1 -  t_p)} (2B^t V^{-1})_j v_i + ((1-t_p)^{-1/2} (I_{p - 1} - V)^{-1/2} (T^t)^{-1})_{i, j} \\
  %          \implies 
            \frac{\partial v}{\partial B} &=- \frac{1}{1 -  t_p} v B^t V^{-1} + (1-t_p)^{-1/2} (I_{p - 1} - V)^{-1/2} (T^t)^{-1}~, 
        \end{align*}
       and it follows
               \begin{align*}
            \det M = \det\bigg(\frac{\partial v}{\partial B} - \frac{\partial v}{\partial c} \cdot \frac{\partial t_p}{\partial B}\bigg) = (1-t_p)^{-(p - 1)/2} (\det(I_{p - 1} - V))^{-(p - 1)/2} (\det V)^{-1/2}.
        \end{align*}
      From \eqref{dist_jacobi_density_u} we obtain that the joint density of $V$, $v$ and $t_p$ 
      is proportional to the function 
              \begin{align*}
            g_{  \gamma ,  \delta}(V)        & = (\det V)^{\gamma - e_1 + 1/2}( \det(I_{p - 1} - V))^{\delta - e_1 + 1/2} \cdot  t_p^{\gamma - e_1}  (1 - t_p)^{\delta - e_1 + (p - 1)/2} \cdot (1 - v^tv)^{\delta - e_1} \\
            &\qquad \times  I\{0_{p - 1} < V < I_{p - 1}\} I\{0 < t_p < 1\} I\{v^tv < 1\} \\
            &= (\det V)^{\gamma - p/2}( \det(I_{p - 1} - V)^{\delta - p/2} \cdot  t_p^{\gamma - (p - 1)/2 - 1}  (1 - t_p)^{\delta - 1} \cdot (1 - v^tv)^{\delta - (p + 1)/2} \\
            &\qquad  \times  I\{0_{p - 1} < V < I_{p - 1}\} I\{0 < t_p < 1\} I\{v^tv < 1\}
        \end{align*}
        Therefore the random variables $V$, $v$ and $t_p$ are independent with $V \sim J\beta E_{p - 1}(\gamma, \delta)$ and $t_p \sim \beta(\gamma - (p-1)/2, \delta)$.  Obviously,
        \begin{align*}
            \frac{\det U}{\det V} &= t_p ~,~~
            \frac{\det(I_p - U)}{\det(I_{p - 1} - V)}  = (1 - t_p)(1 - v^t v),
        \end{align*}
 and the assertion now follows if we can prove $1 - v^t v \sim \beta(\delta - (p - 1)/2, (p - 1)/2)$, or equivalently $v^t v \sim \beta((p - 1)/2, \delta - (p - 1)/2)$.         
        
        To see this, we will apply Lemma 2.1 in \cite{gupson1997}. Since $v$ has a density proportional to $g(\|v\|_2^2)$ with the function $g(x) = (1 - x)^{\delta - (p + 1)/2}I\{x < 1\}$, the density of $\|v\|_2$ is proportional to
        \begin{align*}
            x^{p - 2} g(x^2) I\{0 < x\} = x^{p - 2}(1 - x^2)^{\delta - (p + 1)/2} I\{0 < x < 1\}
        \end{align*}
        Using a simple substitution, the density of $\|v\|_2^2$ is therefore proportional to
        \begin{align*}
            x^{p/2 - 1}(1 - x)^{\delta - (p + 1)/2} I\{0 < x < 1\} x^{-1/2} = x^{(p - 1)/2 - 1} (1 - x)^{\delta - (p - 1)/2 - 1} I\{0 < x < 1\},
        \end{align*}
        i.e. $v^tv = \|v\|_2^2 \sim \beta((p - 1)/2, \delta - (p - 1)/2)$. This concludes the proof of Theorem~\ref{dist_jacobi}.
    \end{proof}

    \begin{thm} \label{jacobi_subblock}
        Assume $M \sim J\beta E_p(\gamma, \delta)$, and let $M^{[j]}$ denote the $j \times j$ upper left subblock of the matrix $M$. 
        There exist independent random variables $p_{0,1}, \ldots , p_{p-1,1}$
        and $p_{1,2}, \ldots , p_{p-1,2}$ with 
        \begin{align} \label{distpropa}
            p_{i, 1} &\sim \beta(\gamma - i/2, \delta)  ~,~i=0,\ldots ,p-1, \\
            p_{i, 2} &\sim \beta(\delta - i/2, i/2) ~,~i=1,\ldots ,p-1, \label{distpropb}
        \end{align} 
        such that for all real $a, b$ the identity
        \begin{align}
            &\log \Big((\det M^{[j]})^a \cdot (\det(I_j-M^{[j]}))^b\Big) \nonumber \\
            ={}& \log(p_{0, 1}^a (1-p_{0, 1})^b) + \sum \limits_{i = 1}^{j - 1}\left\{\log(p_{i, 1}^a (1-p_{i, 1})^b) + \log(p_{i, 2}^b)\right\} \label{ident_det_jacobi}
        \end{align}
        holds for all $1 \le j \le p$ simultaneously. In particular
        \begin{align*}
  \det M =  \prod_{i = 0}^{p - 1} p_{i, 1}.
        \end{align*}
    \end{thm}
    \begin{proof}
        For $1 \le i \le p - 1$ choose
        \begin{align*}
            p_{0, 1} &= \det M^{[1]} ~,~
            p_{i, 1} = \frac{\det M^{[i + 1]}}{\det M^{[i]}} ~,~
            p_{i, 2} = \frac{\det(I_{i + 1} - M^{[i + 1]})}{(1 - p_{i, 1})\det(I_i - M^{[i]})}~,
        \end{align*}
        then the identity (\ref{ident_det_jacobi}) obviously holds. The statements \eqref{distpropa}  and   \eqref{distpropb}  are  now proved by induction with respect to 
         the parameter $p$.    The claim is obviously correct for $p = 1$, since 
        \begin{align*}
            p_{0, 1} = \det M^{[1]} = M^{[1]} \sim J\beta E_1(\gamma, \delta) = \beta(\gamma, \delta) ~. 
        \end{align*}
        holds.
        
        Now  assume that \eqref{distpropa}  and   \eqref{distpropb} are satisfied  for $1, \ldots, p - 1$, then an application of Theorem~\ref{dist_jacobi}
        yields
        \begin{align*}
            M^{[p - 1]} &\sim J\beta E_{p - 1}(\gamma, \delta) ,\\
            p_{p - 1, 1} &\sim \beta(\gamma - (p - 1)/2, \delta) ,\\
            p_{p - 1, 2} &\sim \beta(\gamma - (p - 1)/2, (p - 1)/2),
        \end{align*}
      and these three random variables are independent. Since $M^{[p - 1]}\sim J\beta E_{p - 1}(\gamma, \delta)$  and 
        \begin{align*}
            (M^{[p]})^{[i]} = (M^{[p - 1]})^{[i]} ~~(  i \le p - 1)
        \end{align*}
                it follows from the induction hypothesis that
        \begin{align*}
            p_{i, 1} \sim \beta(\gamma - i/2, \delta) ~,~  i= 0 \ldots p - 2, \\
            p_{i, 2} \sim \beta(\delta - i/2, i/2)  ~,~  i= 1 \ldots p - 2,
        \end{align*}
        and these random variables are independent. It remains to show that $(p_{i, 1})_{i = 0}^{p - 2}$ and $(p_{i, 2})_{i = 1}^{p - 2}$ are also (jointly) independent of $(p_{p - 1, 1}, p_{p - 1, 2})$. This follows directly from the fact that $(p_{i, 1})_{i = 0}^{p - 2}$ and $(p_{i, 2})_{i = 1}^{p - 2}$ can be written as functions of $M^{[p - 1]}$, which is in turn independent of $(p_{p - 1, 1}, p_{p - 1, 2})$.
    \end{proof}
    
    \begin{rem} {\rm 
    A well-known result  in random matrix theory is the Bartlett decomposition [see  \cite{bart1933}], which states that in the Cholesky decomposition of a Wishart-distributed random matrix the entries are independent and normal resp. \hbox{$\chi^2$-distributed}. A corresponding result for the Jacobi-beta-ensemble was derived by \cite{kshi1961}
           and reads as follows. If $X \sim J\beta E_p(\gamma, \delta)$ has the (random) Cholesky decomposition $X = T^t T$ for some upper triangular matrix $T$, then the diagonal entries $t_{11}, \ldots, t_{pp}$ of $T$ are independent and their squares are beta-distributed, that is  $t_{ii}^2 \sim \beta(\gamma - (i - 1)/2, \delta)$. This result 
           is a special case of   Theorem~\ref{jacobi_subblock} that can be obtained for $a = 1$, $b = 0$. 

        To see this, denote by $K_i$ the $p \times i$-matrix with $(K_i)_{j, k} = \delta_{jk}$. Then the equation
        \begin{align*}
            X^{[i]} = K_i^t X K_i = K_i^t T^t K_i K_i^t T K_i = (K_i^tT K_i)^t K_i^tT K_i
        \end{align*}
        holds. Noting $\det(K_i^tT K_i) = t_{11} \cdot \ldots \cdot t_{ii}$ we can conclude that
        \begin{align*}
            t_{ii}^2 = \frac{\det(K_i^tT K_i)^2}{\det(K_{i - 1}^t T K_{i - 1})^2} = \frac{\det(X^{[i]})}{\det(X^{[i - 1]})} = p_{i - 1, 1}
        \end{align*}
        are independent random variables with $t_{ii}^2 \sim \beta(\gamma - (i-1)/2, \delta)$. 
        }
    \end{rem}

    % General convergence result
\section{Random Hankel determinant processes}  \label{section_applied_convergence}

    Let    $(M_{1, 2n + 1}^p, \ldots, M_{2n + 1, 2n + 1}^p)$ denote a uniformly distributed random vector on the $(2n + 1)$-th moment space $\mathcal{M}_{2n + 1}^{p_n}$
 and recall the definition of the stochastic process  $\{H_n (t) \}_{t \in [0,1]} $ in \eqref{basicproc}.
   From \eqref{produkt_h}, \eqref{dist_can_mom} and Theorem~\ref{jacobi_subblock} we obtain  the following representation 
    \begin{align*}
          H_n( t) ={}& \sum \limits_{i = 1}^{\lfloor nt \rfloor}\Big\{
                \sum \limits_{j = 1}^{p_n  - 1} \left(\log(r_{2n + 1, 2i, j}^{\lfloor nt\rfloor - i}) + \log(r_{2n + 1, 2i - 1, j}^{\lfloor nt\rfloor - i + 1}) \right) \\
                &+\sum \limits_{j = 0}^{p_n- 1}\log(p_{2n + 1, 2i, j}^{\lfloor nt\rfloor - i + 1}(1 - p_{2n + 1, 2i, j})^{\lfloor nt\rfloor - i}) \\
                &+\sum \limits_{j = 0}^{p_n  - 1}\log(p_{2n + 1, 2i - 1, j}^{\lfloor nt\rfloor - i + 1}(1 - p_{2n + 1, 2i - 1, j})^{\lfloor nt\rfloor - i + 1})
            \Big\}~,~
    \end{align*}
    where  the random variables $p_{2n + 1, i, j}$ and $r_{2n + 1, i, j}$ are independent and beta-distributed, that is 
    \begin{align}
        p_{2n + 1, i, j} &\sim \beta\Big (\frac{p_n + 1}{2}(2n - i + 2) - j/2, \frac{p_n + 1}{2}(2n - i + 2)\Big ) \label{beta_param_1} \\
        r_{2n + 1, i, j} &\sim \beta\Big (\frac{p_n + 1}{2}(2n - i + 2) - j/2, j/2\Big  ). \label{beta_param_2}
    \end{align}
    In the following discussion, we will consider a more general process. When viewing the Hankel-determinant as a function of the canonical moment matrices $\mathcal{U}_{i, 2n}^{p_n}$ [see equation~\eqref{produkt_h}], we can not only vary the dimension $n$ of the Hankel-matrix, but also the size $p_n$ of the canonical moment matrices in \eqref{produkt_h}. To this extent we introduce a new parameter $s \in (0, 1]$ and consider the upper left $\lfloor p_n s \rfloor \times \lfloor p_n s \rfloor$ subblocks of the canonical moment matrices. According to Theorem~\ref{jacobi_subblock}, the distribution of the logarithm of the corresponding Hankel-determinant can be written as
       \begin{align}
        \label{def_h}
        \begin{split}
          H_n(s, t) ={}&\sum \limits_{i = 1}^{\lfloor nt \rfloor}\Big\{
                \sum \limits_{j = 1}^{\lfloor p_n s \rfloor - 1} \left(\log(r_{2n + 1, 2i, j}^{\lfloor nt\rfloor - i}) + \log(r_{2n + 1, 2i - 1, j}^{\lfloor nt\rfloor - i + 1}) \right) \\
                &+\sum \limits_{j = 0}^{\lfloor p_ns \rfloor - 1}\log(p_{2n + 1, 2i, j}^{\lfloor nt\rfloor - i + 1}(1 - p_{2n + 1, 2i, j})^{\lfloor nt\rfloor - i}) \\
                &+\sum \limits_{j = 0}^{\lfloor p_n s \rfloor - 1}\log(p_{2n + 1, 2i - 1, j}^{\lfloor nt\rfloor - i + 1}(1 - p_{2n + 1, 2i - 1, j})^{\lfloor nt\rfloor - i + 1})
            \Big\}~,~
        \end{split}
    \end{align}
where we use the convention  $H_n(0, t) = 0$ and the process in  \eqref{basicproc} is obtained as $H_n( t) = H_n(1, t)  $. In the following discussion we will investigate the weak convergence of the process $\{H_n(s, t)\}_{s, t \in [0, 1]}$ as $n, p_n \to \infty$.

For this
purpose we state a general result with sufficient conditions for  the weak convergence of  a process of the form \eqref{def_h}, which might be of independent interest.       The proof can be found in the Appendix~\ref{sec6}.
        \begin{thm} \label{limit_thm}
       For each $n\in \N $   let  $T_n$  be a finite set, let $\{ X_n(i)~|~i \in T_n \}$ be real valued random variables  and let  $g_n: T_n \times [0, 1]^k 
       \to \mathbb{R}$ be a  real-valued function. Consider a process of the form
    \begin{align*}
        Z_n(t_1, \ldots, t_k) = \sum \limits_{i \in T_n} g_n(i, t_1, \ldots, t_k) X_n(i) 
            \end{align*}
 and suppose that the following assumptions are satisfied
    \begin{enumerate}[label=\textbf{(C\arabic*)}]
        \item The random variables $(X_n(i))_{i \in T_n}$ are independent. \label{cond_indep}
        \item $g_n$ is right-continuous in each of the last $k$ components. \label{cond_cont}
        \item $\e{X_{n}(i)} = 0$. \label{cond_center}
        \item $\e{X_{n}^{2k + 2}(i)} \le C \e{X_{n}^2(i)}^{k + 1} < \infty$ for some universal constant $C > 1$. \label{cond_moments}
        \item $\sup \limits_{i \in T_n} g_n^2(i, t) \var{X_n(i)} \xrightarrow{n \to \infty} 0$ for all fixed $t \in [0, 1]^k$. \label{cond_var_bound}
        \item There exists a function $f:[0, 1]^{2k} \to \mathbb{R}$ such that for all $s, t \in [0, 1]^k$ \label{cond_var_conv}
            \begin{align*}
                \cov{Z_n(s)}{Z_n(t)} = \sum \limits_{i \in T_n} g_n(i, s)g_n(i, t)\var{X_n(i)} \xrightarrow{n \to \infty} f(s, t). 
            \end{align*}
        \item There are sequences $h_n^{(j)} \xrightarrow{n \to \infty} 0$ for $1 \le j \le k$ such that for any two vectors $T = (t_1, \ldots, t_k)$ and $T' = (t_1, \ldots, t_j', \ldots, t_k)$ the inequality
            \begin{align*}
                \var{Z_n(T) - Z_n(T')} \le C(|t_j - t_j'| + h_n^{(j)})
            \end{align*}
            holds.\label{cond_cov}
        \item For any $t_j \le t_j' \le t_j'' \le t_j + h_n^{(j)}$ at least one of the equations
            \begin{align*}
                g_n(i, t_1, \ldots, t_j, \ldots, t_k) &= g_n(i, t_1, \ldots, t_j', \ldots, t_k)\\
                &\text{or}\\
                g_n(i, t_1, \ldots, t_j', \ldots, t_k) &= g_n(i, t_1, \ldots, t_j'', \ldots, t_k)
            \end{align*}
            holds. \label{cond_lattice}
    \end{enumerate}
  Then the process $Z_n$ converges weakly in $l^\infty([0, 1]^k)$ to a centered, continuous Gaussian process with covariance kernel $f$.
    \end{thm}

    \begin{thm}\label{convergence_hn}
        If $p_n \xrightarrow{n \to \infty} \infty$, then
        \begin{align*}
            \{ \widetilde{H}_n(s, t) \}_{s,t \in [0,1]} := \frac{1}{\sqrt{n}} \big \{ H_n (s, t) - \mathbb{E}[H_n(s, t) ]\big  \}_{s,t \in [0,1]} \implies \{ \mathcal{G}(s, t) \}_{s,t \in [0,1]}
        \end{align*}
      in $l^\infty([0, 1]^2)$,  where  $\{ H_n(s, t) \}_{s,t \in [0,1]} $ is the process defined in (\ref{def_h})
                and $\mathcal{G}$ is a centered continuous Gaussian process with covariance kernel
        \begin{align*}
            \cov{\mathcal{G}(s_1, t_1)}{\mathcal{G}(s_2, t_2)} = \frac{(s_1 \wedge s_2)^2 c(t_1, t_2)}{2}.
        \end{align*}
        The function $c$ is given by
        \begin{align*}
            c(t_1, t_2) = \int_0^{t_1 \wedge t_2} \frac{(t_1 - x)(t_2 - x)}{(1 - x)^2} dx.
        \end{align*}
    \end{thm}
    \begin{proof}
        We use the decomposition 
                \begin{align*}
            \widetilde{H}_n = A_n + B_n + C_n + D_n + E_n
        \end{align*}
 of  the process $\widetilde{H}_n$,       where the processes on the right-hand side are defined by 
        \begin{align*}
            A_n ={}& \frac{1}{\sqrt{n}}\sum \limits_{i = 1}^{\lfloor nt \rfloor - 1}\sum \limits_{j = 1}^{\lfloor p_ns \rfloor - 1}(\lfloor nt \rfloor - i)\left(\log(r_{2n + 1, 2i, j}) - \e{\log(r_{2n + 1, 2i, j})}\right) \\
            B_n ={}& \frac{1}{\sqrt{n}}\sum \limits_{i = 1}^{\lfloor nt \rfloor}\sum \limits_{j = 1}^{\lfloor p_ns \rfloor - 1}(\lfloor nt \rfloor - i + 1)\left(\log(r_{2n + 1, 2i - 1, j}) - \e{\log(r_{2n + 1, 2i - 1, j})}\right) \\
            C_n ={}& \frac{1}{\sqrt{n}}\sum \limits_{i = 1}^{\lfloor nt \rfloor - 1}\sum \limits_{j = 0}^{\lfloor p_ns \rfloor - 1} (\lfloor nt \rfloor - i)\big\{\log(p_{2n + 1, 2i, j}(1 - p_{2n + 1, 2i, j}))\\
            &\qquad- \e{\log(p_{2n + 1, 2i, j}(1 - p_{2n + 1, 2i, j}))}\big\} \\
            D_n ={}& \frac{1}{\sqrt{n}}\sum \limits_{i = 1}^{\lfloor nt \rfloor}\sum \limits_{j = 0}^{\lfloor p_ns \rfloor - 1} (\lfloor nt \rfloor - i + 1)\big\{\log(p_{2n + 1, 2i - 1, j}(1 - p_{2n + 1, 2i - 1, j})) \\
            &\qquad- \e{\log(p_{2n + 1, 2i - 1, j}(1 - p_{2n + 1, 2i - 1, j}))}\big\} \\
            E_n ={}&\frac{1}{\sqrt{n}} \sum \limits_{i = 1}^{\lfloor nt \rfloor}\sum \limits_{j = 0}^{\lfloor p_ns \rfloor - 1} \left(\log(p_{2n + 1, 2i, j}) 
            - \e{\log(p_{2n + 1, 2i, j})}\right)
        \end{align*}
   and the random variables $p_{2n+1,i,j} , r_{2n+1,i,j}$ are independent and beta-distributed, as specified in \eqref{beta_param_1} and \eqref{beta_param_2} respectively. We will now apply Theorem \ref{limit_thm} to each of these processes to prove
        \begin{align}
            A_n &\implies \mathcal{G}'   \label{an}\\
            B_n &\implies \mathcal{G}'  \label{bn} \\
            C_n &\implies 0  \label{cn}  \\
            D_n &\implies 0   \label{dn}\\
            E_n &\implies 0,  \label{en}
        \end{align}
        where $\mathcal{G}'$ is a continuous centered Gaussian process with covariance kernel
        \begin{align*}
            \cov{\mathcal{G}'(s_1, t_1)}{\mathcal{G}'(s_2, t_2)} = \frac{(s_1 \wedge s_2)^2 c(t_1, t_2)}{4}.
        \end{align*}
        Since $A_n$ and $B_n$ are independent  it follows that $                   (A_n, B_n) \implies (\mathcal{G}', \mathcal{G}'') $, where $\mathcal{G}''$ is an 
        independent copy of $\mathcal{G}'$ (c.f. Example 1.4.6 in \cite{vanwell1996}),  and the continuous mapping theorem 
        implies $A_n + B_n \implies \mathcal{G}' + \mathcal{G}'' \stackrel{d}{=} \mathcal{G}$. The assertion of Theorem \ref{convergence_hn}  now follows from 
         Slutsky's lemma.  
        We will omit the proof of \eqref{bn}  and \eqref{dn} because  the arguments are similar as in the proof of 
    \eqref{an}  and \eqref{cn}, respectively.  \br 
        
     {   \bf Proof of \eqref{an}:}     We can represent the process  $A_n$ in a  form such that  Theorem \ref{limit_thm} is aplicable, that is 
        \begin{align} \label{hol1}
            A_n(s, t) = \sum \limits_{(i,j)  \in T_n} g_n((i,j) , s, t) X_n((i,j) )~, 
        \end{align}
        where
        \begin{align}
            T_n &= \{1, \ldots, n - 1\} \times \{1, \ldots, p_n - 1\}  \nonumber \\
            g_n((i, j), s, t) &= \frac{1}{\sqrt{n}} (\lfloor nt \rfloor - i) I\{i \le \lfloor nt \rfloor - 1\} I\{j \le \lfloor p_n s\rfloor - 1\} 
            \nonumber \\
            X_n((i, j)) &= \log(r_{2n + 1, 2i, j}) - \e{\log(r_{2n + 1, 2i, j})} \label{hol2}
        \end{align}
        It is obvious that $A_n$ satisfies the conditions \ref{cond_indep}, \ref{cond_cont} and \ref{cond_center}. Condition \ref{cond_moments} is proved in Theorem \ref{mom_high} in Appendix~\ref{betamom} (note that the parameters of the distribution of $r_{2n + 1, i, j}$ are bounded from below by $\tfrac{1}{2}$). 
        
        By \eqref{form_kappa1} from Appendix~\ref{betamom} the variance of the logarithm of a beta distributed  random variable $X \sim \beta(a, b)$ can be calculated as
        \begin{align*}
            \var{\log X} = \psi_1(a) - \psi_1(a + b),
        \end{align*}
        where
        \begin{equation} \label{polgma}
        \psi_k(x) = \frac{d^{k + 1}}{dx^{k + 1}} \log \Gamma(x)
        \end{equation}
denotes the Polygamma function. 
     An application of formula (\ref{bound_psi_diff}) from Appendix \ref{appendix_polygamma}  shows  that
        \begin{align*}
            &g_n^2((i, j), s, t) \var{X_n(i, j)} \\
            \le{}&  \frac{3 (n - i)^2j/2}{n \left(\frac{p_n + 1}{2}(2n - 2i + 2) - j/2 \right)\frac{p_n + 1}{2}(2n - 2i + 2)} \\
            \le{}&   \frac{3(n - i)^2j/2}{n \left((p_n + 1)(n - i + 1/2)  + (p_n + 1 - j)/2 \right)( p_n + 1) (n - i + 1)} \\
            \le{}& \frac{3j/2}{n \left(p_n + 1\right)^2} \le  \frac{3}{n \left(p_n + 1\right)} \xrightarrow{n \to \infty} 0~,
        \end{align*}
   and therefore  condition \ref{cond_var_bound} is also satisfied.   
  To see that condition \ref{cond_var_conv} holds,    define for  $s_1, s_2, t_1, t_2 \in [0, 1]$  the minima  $s = s_1 \wedge s_2$ and $t = t_1 \wedge t_2$. By (\ref{var_first_order}) we have a decomposition 
        \begin{align*}
            \cov{A_n(s_1, t_1)}{A_n(s_2, t_2)} = S_n + R_n,
        \end{align*}
        where
        \begin{align*}
            S_n = \sum \limits_{i = 1}^{\lfloor nt \rfloor - 1} \sum \limits_{j = 1}^{\lfloor p_n s\rfloor - 1} \frac{1}{n}\frac{(\lfloor nt_1\rfloor - i) (\lfloor nt_2\rfloor - i) j/2}{\left(\frac{p_n + 1}{2} (2n - 2i + 2) - j/2\right) \frac{p_n + 1}{2} (2n - 2i + 2)},
        \end{align*}
        and the remainder $R_n$ satisfies the inequality
        \begin{align*}
            |R_n| \le \sum \limits_{i = 1}^{\lfloor nt \rfloor - 1} \sum \limits_{j = 1}^{\lfloor p_n s \rfloor - 1} \frac{4}{n} \frac{(\lfloor nt_2 \rfloor - i)(\lfloor nt_2 \rfloor - i)}{\left(\frac{p_n + 1}{2}(2n - 2i + 2) - j/2\right)^2}.
        \end{align*}
        Now note that
        \begin{align*}
            S_n ={}& \sum \limits_{i = 1}^{\lfloor nt \rfloor - 1} \sum \limits_{j = 1}^{\lfloor p_n s\rfloor - 1} \frac{1}{n}\frac{(\lfloor nt_1\rfloor - i) (\lfloor nt_2\rfloor - i) j/2}{\left(\frac{p_n + 1}{2} (2n - 2i + 2) - j/2\right) \frac{p_n + 1}{2} (2n - 2i + 2)} \\
            \ge{}&\sum \limits_{i = 1}^{\lfloor nt \rfloor - 1} \sum \limits_{j = 1}^{\lfloor p_n s\rfloor - 1} \frac{1}{n}\frac{(nt_1 - i - 1) (nt_2 - i - 1)j/2}{\left(\frac{p_n + 1}{2} (2n - 2i + 2)\right)^2} \\
            ={}&\frac{1}{n}
            \sum \limits_{j = 1}^{\lfloor p_n s \rfloor - 1} \frac{j}{2(p_n + 1)^2}   \sum \limits_{i = 1}^{\lfloor nt \rfloor - 1}\frac{(nt_1 - i - 1) (nt_2 - i - 1)}{(n - i + 1)^2} \\
            ={}& \frac{1}{4} (s^2 + O(p_n^{-1}))\left(c(t_1, t_2) + O(n^{-1})\right) \xrightarrow{n \to \infty} \frac{(s_1 \wedge s_2)^2 c(t_1, t_2)}{4} .
        \end{align*}
Moreover, 
        \begin{align*}
            S_n ={}&\sum \limits_{i = 1}^{\lfloor nt\rfloor - 1} \sum \limits_{j = 1}^{\lfloor p_n s \rfloor - 1}\frac{1}{n} \frac{(\lfloor nt_1\rfloor - i) (\lfloor nt_2\rfloor - i) j/2}{\left(\frac{p_n + 1}{2} (2n - 2i + 2) - j/2\right) \frac{p_n + 1}{2} (2n - 2i + 2)} \\
            \le{}&\sum \limits_{i = 1}^{\lfloor nt \rfloor - 1} \sum \limits_{j = 1}^{\lfloor p_n s \rfloor - 1} \frac{1}{n}\frac{(nt_1 - i) (nt_2 - i)j/2}{\left(\frac{p_n + 1}{2} (2n - 2i)\right)^2} \\
            ={}& \Big (\frac{s^2}{4} + O(p_n^{-1})\Big )\left(c(t_1, t_2) + O(n^{-1})\right) \xrightarrow{n \to \infty} \frac{(s_1 \wedge s_2)^2 c(t_1, t_2)}{4}~,
        \end{align*}
       which gives 
        \begin{align*}
            S_n \xrightarrow{n \to \infty} \frac{(s_1 \wedge s_2)^2 c(t_1, t_2)}{4}.
        \end{align*}
Finally, the remainder $R_n$ vanishes asymptotically, that is
        \begin{align*}
            |R_n| \le{}&\sum \limits_{i = 1}^{\lfloor nt \rfloor - 1} \sum \limits_{j = 1}^{\lfloor p_n s \rfloor - 1} \frac{4}{n} \frac{(\lfloor nt_2 \rfloor - i)(\lfloor nt_2 \rfloor - i)}{\left(\frac{p_n + 1}{2}(2n - 2i + 2) - j/2\right)^2} \\
            \le{}& \sum \limits_{i = 1}^{\lfloor nt \rfloor - 1} \sum \limits_{j = 1}^{\lfloor p_n s \rfloor - 1}\frac{4}{n} \frac{(n - i)^2}{((p_n + 1)(n - i + 1/2))^2} 
            \le{} \frac{4}{p_n} \xrightarrow{n \to \infty} 0,
        \end{align*}       
        which proves condition \ref{cond_var_conv}. 
        To show   \ref{cond_cov} note that
        \begin{align}
             \var{\log r_{2n, 2i, j}} 
            \le{}  5 \frac{j/2}{\left(\frac{p_n + 1}{2}(2n - 2i + 2) - j/2\right)^2} 
            \le{}& \frac{5j}{((p_n + 1)(n - i))^2} \label{u_bnd_1}
        \end{align}
        holds by the upper bound (\ref{bound_psi_diff}) in Appendix \ref{appendix_polygamma}.     Assume now that $0 \le t_1 < t_2 \le 1$ and $0 \le s \le 1$ are real numbers. Then for $X = (s, t_1)$ and $Y = (s, t_2)$ 
        we have         
        \begin{align*}
            \var{Z_n(Y) - Z_n(X)}
                %            ={}& \sum \limits_{i = 1}^{\lfloor nt_1 \rfloor - 1} \sum \limits_{j = 1}^{\lfloor p_n s\rfloor - 1} \frac{1}{n}(\lfloor nt_2 \rfloor - \lfloor nt_1 \rfloor)^2\var{\log r_{2n, 2i, j}} \\
   %         &+ \sum \limits_{i = \lfloor nt_1 \rfloor}^{\lfloor n t_2\rfloor - 1} \sum \limits_{j = 1}^{\lfloor p_n s\rfloor - 1} \frac{1}{n}(\lfloor nt_2 \rfloor - \lfloor nt_1 \rfloor)^2\var{\log r_{2n, 2i, j}}\\
                %
            \le{}& \sum \limits_{i = 1}^{\lfloor nt_1 \rfloor - 1} \sum \limits_{j = 1}^{\lfloor p_n s\rfloor - 1} \frac{5}{n}(\lfloor nt_2 \rfloor - \lfloor nt_1 \rfloor)^2\frac{j}{(p_n + 1)^2 (n - i)^2} \\
            &+ \sum \limits_{i = \lfloor nt_1 \rfloor}^{\lfloor n t_2\rfloor - 1} \sum \limits_{j = 1}^{\lfloor p_n s\rfloor - 1} \frac{5}{n}(\lfloor nt_2 \rfloor - i)^2 \frac{j}{(p_n + 1)^2 (n - i)^2} \\
            \le{}& \sum \limits_{i = 1}^{\lfloor nt_1 \rfloor - 1} \frac{5}{n}(\lfloor nt_2 \rfloor - \lfloor nt_1 \rfloor)^2 \frac{1}{(n - i)^2}
            + \sum \limits_{i = \lfloor nt_1 \rfloor}^{\lfloor n t_2\rfloor - 1} \frac{5}{n}(\lfloor nt_2 \rfloor - i)^2 \frac{1}{(n - i)^2} \\
            \le{}& \frac{(\lfloor nt_2 \rfloor - \lfloor nt_1 \rfloor)^2}{n^2}\sum \limits_{i = 1}^{\lfloor n t_1 \rfloor - 1} \frac{5}{n}\frac{1}{(1 - i/n)^2} + \sum \limits_{i = \lfloor nt_1 \rfloor}^{\lfloor n t_2\rfloor - 1} \frac{5}{n} \\
                %
 %           \le{}& 5\frac{(\lfloor nt_2 \rfloor - \lfloor nt_1 \rfloor)^2}{n^2}\int \limits_{0}^{\frac{\lfloor nt_1\rfloor}{n}} \frac{1}{(1-x)^2} dx + 5 \left(t_2 - t_1 + \frac{1}{n}\right) \\
                %
  %          ={}& 5\frac{(\lfloor nt_2 \rfloor - \lfloor nt_1 \rfloor)^2}{n^2}\frac{\lfloor nt_1\rfloor}{n - \lfloor nt_1 \rfloor} + 5 \left(t_2 - t_1 + \frac{1}{n}\right) \\
                %
            \le{}& 10 \Big (t_2 - t_1 + \frac{1}{n}\Big ).
        \end{align*}
        For the increments in the second coordinate, let $0 \le s_1 < s_2 \le 1$ and $0 \le t \le 1$ be real numbers and set $X = (s_1, t)$, $Y = (s_2, t)$, then
        \begin{align*}
            \var{Z_n(Y) - Z_n(X)}
            & =\sum \limits_{i = 1}^{\lfloor n t\rfloor - 1} \sum \limits_{j = \lfloor p_n s_1 \rfloor}^{\lfloor p_n s_2\rfloor - 1} \frac{1}{n}(\lfloor nt \rfloor - i)^2\var{\log r_{2n, 2i, j}}\\
         &    \le{}\sum \limits_{i = 1}^{\lfloor n t\rfloor - i} \sum \limits_{j = \lfloor p_n s_1 \rfloor}^{\lfloor p_n s_2\rfloor - 1} \frac{5}{n}\frac{j}{(p_n + 1)^2}
            \le \frac{5}{p_n + 1}\sum\limits_{j = \lfloor p_n s_1 \rfloor}^{\lfloor p_n s_2\rfloor - 1} 1\\
          &  \le{} 5\Big (s_2 - s_1 + \frac{1}{p_n}\Big ).
        \end{align*}
        Therefore condition  \ref{cond_cov} is also satisfied with $h_n^{(1)} = \frac{1}{p_n}$ and $h_n^{(2)} = \frac{1}{n}$. It is obvious that \ref{cond_lattice} holds for these sequences and the assertion \eqref{an}  follows from Theorem \ref{limit_thm}. \br

    {\bf Proof of    \eqref{cn}:}
    The process     $C_n$ can be decomposed in a similar way as $A_n$ in \eqref{hol1}, where the random variables
    $X_n((i,j) )$  in \eqref{hol2} are now defined by
   \begin{align*}
        X_n((i,j) ) = \log\big(p_{2n + 1, 2i, j}(1 - p_{2n + 1, 2i, j})\big) - \e{\log\big(p_{2n + 1, 2i, j}(1 - p_{2n + 1, 2i, j})\big)}.
   \end{align*}
   
   It is again obvious that  \ref{cond_indep}, \ref{cond_cont} and \ref{cond_center} hold. Condition \ref{cond_moments} 
is a consequence of Theorem \ref{mom_high_2} in Appendix \ref{betamom}. 
        
        For a  proof of   \ref{cond_var_bound} note that  by  formula (\ref{form_kappa2})  in Appendix \ref{betamom} 
        the variance of $\log(X(1-X))$ for a    random variable $X \sim \beta(a, b)$  can be calculated as
        \begin{align*}
            \var{\log(X(1-X))} = \psi_1(a) + \psi_1(b) - 4 \psi_1(a + b)~. 
        \end{align*}
An application of (\ref{bnd_var}) now  gives 
        \begin{align}
            \var{\log(p_{2n + 1, 2i, j}(1 - p_{2n + 1, 2i, j}))} 
            \le{}& \Big(6 + \frac{(j/2)^2}{\left(\frac{p_n + 1}{2}(2n - 2i + 2) - \frac{j}{2}\right)}\Big) \frac{1}{\left(\frac{p_n + 1}{2}(2n - 2i + 2)\right)^2} \nonumber\\
            \le{}& \Big (6 + \frac{j^2}{(p_n + 1)(n - i)}\Big ) \frac{1}{((p_n + 1)(n - i))^2} \nonumber\\
            ={}& \frac{6}{((p_n + 1)(n - i))^2} + \frac{j^2}{((p_n + 1)(n - i))^3} \nonumber\\
            \le{}& \frac{6(j + 1)}{((p_n + 1)(n - i))^2}, \label{u_bnd_2}
        \end{align}
    and condition        \ref{cond_var_bound} follows from the inequality        
        \begin{align*}
            \frac{1}{n} (\lfloor nt\rfloor - i)^2\var{\log(p_{2n + 1, 2i, j}(1 - p_{2n + 1, 2i, j})}    \le{}
            \frac{1}{n} (\lfloor nt\rfloor - i)^2 \frac{6(j + 1)}{((p_n + 1)(n - i))^2} \le{}\frac{6}{np_n} .
            %\xrightarrow{n \to \infty} 0
        \end{align*}
    For a proof of assumption     \ref{cond_var_conv} note that 
        \begin{align*}
            \var{C_n(s, t)} 
  %          ={}&\sum \limits_{i = 1}^{\lfloor nt\rfloor - 1} \sum \limits_{j = 0}^{\lfloor p_n s \rfloor - 1}\frac{1}{n} (\lfloor nt \rfloor - i)^2 \var{\log(p_{2n + 1, 2i, j}(1 - p_{2n + 1, 2i, j})} \\
                %
            \le{}& \sum \limits_{i = 1}^{n - 1} \sum \limits_{j = 0}^{p_n - 1}\frac{1}{n} \Big (6 + \frac{j^2}{((p_n + 1)(n - i))^2}\Big ) \frac{1}{(p_n + 1)^2} \\
            \le{}& \frac{6}{p_n + 1} + \sum \limits_{i = 1}^{n - 1} \sum \limits_{j = 0}^{p_n - 1}\frac{1}{n} \frac{1}{(p_n + 1)i} \frac{j^2}{(p_n+ 1)^2} \\
            \le{}& \frac{6}{p_n + 1} + \frac{\log(n) + 1}{n} \xrightarrow{n \to \infty} 0.
        \end{align*}
        By the Cauchy-Schwarz-inequality this implies
        \begin{align*}
            \cov{C_n(s_1, t_1)}{C_n(s_2, t_2)} \xrightarrow{n \to \infty} 0.
        \end{align*}
   Assumption          \ref{cond_cov} follows by similar calculations as given for the proof of \eqref{an} (note the similarity between (\ref{u_bnd_1}) and (\ref{u_bnd_2})). More specifically, the inequalities 
        \begin{align*}
            \var{C_n(s, t_1) - C_n(s, t_2)} &\le 24 \Big (|t_2 - t_1| + \frac{1}{n}\Big ) \\
            \var{C_n(s_1, t) - C_n(s_2, t)} &\le 12 \Big( |s_2 - s_1| + \frac{1}{p_n}\Big )
        \end{align*}
        hold for any real numbers $s, t \in [0, 1]$, $0 \le s_1 < s_2 \le 1$ and $0 \le t_1 \le t_2 \le 1$. Finally 
        condition    \ref{cond_lattice} is obvious with $h_n^{(1)} = \frac{1}{p_n}$, $h_n^{(2)} = \frac{1}{n}$  and the assertion \eqref{cn} follows 
        from Theorem  \ref{limit_thm}. \br 
        
      {\bf Proof of \eqref{en}:}   We use again Theorem  \ref{limit_thm} to prove the assertion.
       Conditions \ref{cond_indep} -- \ref{cond_moments} hold by similar arguments as in the proof of $A_n$. 
     Observing the inequality    
        \begin{align}
            \var{\log(p_{2n + 1, 2i, j}) } % - \e{\log(p_{2n + 1, 2i, j})}} 
            \le{}& 5 \frac{\frac{p_n + 1}{2}(2n - 2i + 2)}{\left(\frac{p_n + 1}{2}(2n - 2i + 2) - j/2\right)\left((p_n + 1)(2n - 2i + 2) - j/2\right)}\nonumber\\
            \le{}& 5 \frac{(p_n + 1)(n - i + 1)}{\left((p_n + 1)(n - i + 1/2)\right)\left((p_n + 1)(2n - 2i + 1)\right)} \nonumber\\
            \le{}& 5 \frac{n - i + 1}{(p_n + 1)(n - i + 1/2)^2} \le \frac{5}{(p_n + 1)(n - i + 1/2)} \label{bnd_rem}
        \end{align}
   (which follows from      (\ref{bound_psi_diff})) condition   \ref{cond_var_bound} is obviously satisfied.   
        To prove \ref{cond_var_conv}, note that it follows from (\ref{bnd_rem})
        \begin{align*}
            \var{E_n(s, t)} \le \frac{1}{n}\sum \limits_{i = 1}^{n} \sum \limits_{j = 0}^{p_n - 1} \frac{5}{(p_n + 1)(n - i + 1/2)} = O\left(\frac{\log(n)}{n}\right) \xrightarrow{n \to \infty} 0.
        \end{align*}
        Again, by the Cauchy-Schwarz-inequality, this implies
        \begin{align*}
            \cov{E_n(s_1, t_1)}{E_n(s_2, t_2)} \xrightarrow{n \to \infty} 0.
        \end{align*} 
For a proof of      \ref{cond_cov} let $0 \le s_1 < s_2 \le 1$ and $0 \le t \le 1$, then we obtain for 
$X = (s_1, t)$ and $Y = (s_2, t)$ from   (\ref{bnd_rem}) the estimate
        \begin{align*}
            \var{E_n(X) - E_n(Y)} \le \frac{1}{n} \sum \limits_{i = 1}^{\lfloor nt \rfloor}\sum \limits_{j = \lfloor p_n s_1\rfloor}^{\lfloor p_n s_2 \rfloor - 1}\frac{5}{(p_n + 1)(n - i + 1/2)} \le 10 \Big(s_2 - s_1 + \frac{1}{p_n}\big )~,
        \end{align*}
        and similarly for  $0 \le t_1 < t_2 \le 1$, $0 \le s \le 1$,  $X = (s, t_1)$ and $Y = (s, t_2)$ 
        \begin{align*}
            \var{E_n(X) - E_n(Y)} \le \frac{1}{n} \sum \limits_{i = \lfloor nt_1\rfloor + 1}^{\lfloor nt_2 \rfloor}\sum \limits_{j = 0}^{\lfloor p_n s \rfloor - 1}\frac{5}{(p_n + 1)(n - i + 1/2)} \le 10 \Big (t_2 - t_1 + \frac{1}{n}\Big ).
        \end{align*}
      Finally  assumption    \ref{cond_lattice} is obvious with $h_n^{(1)} = \frac{1}{p_n}$ and $h_n^{(2)} = \frac{1}{n}$.
    \end{proof}

    \begin{cor}
        If $p_n \xrightarrow{n \to \infty} \infty$, then
        \begin{align*}
            \{ \widetilde{H}_n( t) \}_{t \in [0,1]} := \frac{1}{\sqrt{n}} \big \{ H_n ( t) - \mathbb{E}[H_n(t) ]\big  \}_{t \in [0,1]} \implies \{ \mathcal{G} (t) \}_{t \in [0,1]}  ,
        \end{align*}
       in $l^\infty([0, 1])$   where $\{ H_n( t) \}_{t \in [0,1]} $ is the process defined in (\ref{basicproc}) 
        and $\{ \mathcal{G}( t) \}_{t \in [0,1]} $ is a centered continuous Gaussian process with covariance kernel
        \begin{align*}
            \cov{\mathcal{G}( t_1)}{\mathcal{G}( t_2)} = \frac{1}{2}\int_0^{t_1 \wedge t_2} \frac{(t_1 - x)(t_2 - x)}{(1 - x)^2} dx.
        \end{align*}
    \end{cor}
    
    \medskip
We conclude with a Glivenko-Cantelli type Theorem.  \medskip
    
    \begin{thm} \label{conv_as}
        If $p_n \xrightarrow{n \to \infty} \infty$, we have for any $\varepsilon > 0$
        \begin{align*}
            n^{-\varepsilon} \|\widetilde H_n\| \xrightarrow{a.s.} 0.
        \end{align*}
    \end{thm}
    \begin{proof}
        We will prove a more general result, namely that for a process $Z_n$ satisfying assumptions \ref{cond_indep} -- \ref{cond_lattice} we have $a_n \|Z_n\|_{\infty} \xrightarrow{a.s.} 0$ for all sequences $a_n$ that satisfy $a_n = O(n^{-(1 + \delta)/(2k + 2)})$ for some $\delta > 0$. In order to prove this, note that by (\ref{ineq_dev}) and (\ref{ineq_fn}) we can apply Theorem~\ref{sup_ineq} from Appendix~\ref{sec6} to conclude that
        \begin{align*}
            \p{\sup \limits_{s, k \in [0, 1]^k} |Z_n(s) - Z_n(t)| > \lambda} \le C'' \lambda^{-(2k+2)}
        \end{align*}
        holds for all $\lambda > 0$, where $C''$ is some constant independent of $\lambda$. Combined with \ref{cond_moments} and \ref{cond_var_conv} this yields
        \begin{align*}
            &\sum \limits_{n = 1}^\infty \p{a_n \|Z_n\|_\infty > \lambda}\\
            \le{}& \sum \limits_{n = 1}^\infty \p{\sup \limits_{s, t \in [0, 1]^k} |Z_n(s) - Z_n(t)| > a_n^{-1} \lambda /2} + \p{|Z_n(0)| > a_n^{-1}\lambda /2} \\
            \le{}& \sum \limits_{n = 1}^\infty \left(C'' + C(f(0, 0) + o(1))^{k + 1}\right)a_n^{2k + 2}(\lambda/2)^{-(2k + 2)} \\
            ={}& \left(C'' + C(f(0, 0) + o(1))^{k + 1}\right)(\lambda/2)^{-(2k + 2)} \sum \limits_{n = 1}^\infty O(n^{-(1+\delta)}) < \infty.
        \end{align*}
        By the Borel-Cantelli Lemma, this implies $a_n \|Z_n\|_\infty \xrightarrow{a.s.} 0$.
        
        To prove the statement of Theorem~\ref{conv_as}, set $k = \lceil\varepsilon^{-1} + 1 \rceil$, $\delta = 1$ and $a_n = n^{-1/(k+1)}$. We can now define a $k$-dimensional partial sum process $\widehat{H}_n$ via
        \begin{align*}
            \widehat{H}_n(s_1, \ldots, s_k) = \widetilde{H}_n(s_1, s_2).
        \end{align*}
        Similar to the proof of Theorem~\ref{convergence_hn}, we can decompose $\widehat{H}_n$ into a sum of processes that satisfy the conditions \ref{cond_indep} -- \ref{cond_lattice} and from the previous result we can conclude
        \begin{align*}
            n^{-\varepsilon} \|\widetilde{H}_n\|_\infty \le a_n \|\widehat{H}_n\|_\infty \xrightarrow{a.s.} 0.
        \end{align*}
    \end{proof}
    
    \begin{cor}[Law of large numbers] \label{thm_lln}
        If $p_n \xrightarrow{n \to \infty} \infty$, we have
        \begin{align*}
            \frac{H_n(s, t)}{n p_n} \xrightarrow[\rm uniformly]{a.s.} -\frac{s^2}{2}(t + (1 - t) \log(1 - t))
        \end{align*}
        for $n \to \infty$.
    \end{cor}
    \begin{proof}
        By Theorem~\ref{conv_as} we know 
        \begin{align*}
            \frac{\| H_n - \e{H_n}\|_\infty}{n p_n} = p_n^{-1} n^{-1/2} \|\tilde{H}_n\|_\infty \xrightarrow{a.s.} 0
        \end{align*}
        The assertion now follows from Lemma~\ref{lemma_cumulant1} and Lemma~\ref{lemma_cumulant2} of the following section, which yield
        \begin{align*}
            \frac{\e{H_n}}{n p_n} \xrightarrow[\text{uniformly}]{n \to \infty} -\frac{s^2}{2} \int_0^t \frac{t - x}{1 - x} \, dx = -\frac{s^2}{2}(t + (1 - t)\log(1-t)).
        \end{align*}
    \end{proof}

    % Mod-Phi/ Large deviations
\section{Mod-\texorpdfstring{$\phi$}{Phi}-convergence, moderate and large deviations} \label{section_mod_phi}
    In this section we study further stochastic properties of the random variables  $H_n(s, t)$  defined in \eqref{def_h}.
    We are particularly interested in the recently introduced concept of mod-$\phi$-convergence [see  \cite{femeni2016}] 
    and large deviation properties.  The different limiting results for the sequence  $a_n^{-1} \big(H_n(s, t) - \e{H_n(s, t)}\big)$ 
    obtained in this section and in Section~\ref{section_applied_convergence} are summarised in Figure \ref{fig1}.
    
   To be precise consider the strip
        \begin{align*}
            S_{(a, b)} = \{z \in \mathbb{C} \mid a < \Re(z) < b\},
        \end{align*}
where  $a < 0 < b$ and let $\phi$ be a non-constant infinitely divisible distribution with moment generating function
$ \exp(\eta(z))$. Let  $(t_n)_{n \in \N } $  denote a real valued sequence and $(X_n)_{n \in \N } $  be a sequence 
of real-valued random variables with existing  moment-generating functions on $S_{(a, b)}$ 
such that $t_n \to \infty$  and that for some non-vanishing analytic function $\psi$  on $S_{(a, b)}$
     \begin{align*}
            \exp(-t_n \eta(z)) \e{\exp(z X_n)} \to \psi(z)
        \end{align*}
        holds locally uniform on $S_{(a, b)}$.  
        Following   \cite{femeni2016} the  sequence $(X_n)_{n \in \mathbb{N}}$  is said to converge mod-$\phi$ on $S_{(a, b)}$ with speed $(t_n)_{n \in \mathbb{N}}$.

     Mod-$\phi$-convergence is a very strong mode of convergence that implies the asymptotic behaviour of $(X_n)_{n \in \mathbb{N}}$ at different scales. Most prominently, Berry-Esseen bounds and large deviation results can be derived from mod-$\phi$-convergence. Particularly the large deviation results are stronger than the results that are usually obtained by a large deviation principle. The former gives an asymptotic equivalent for the probability $\p{X_n \ge t_n x}$, while the latter only yields the limiting behaviour for the logarithm of the probability.
    
   The core idea behind mod-$\phi$-convergence is that the distribution of $X_n$ is close to the distribution of the sum of $t_n$ i.i.d. $\phi$-distributed random variables. The function $\psi(z)$ measures the error made in this approximation and yields further refinement in the asymptotic formulas. \br
    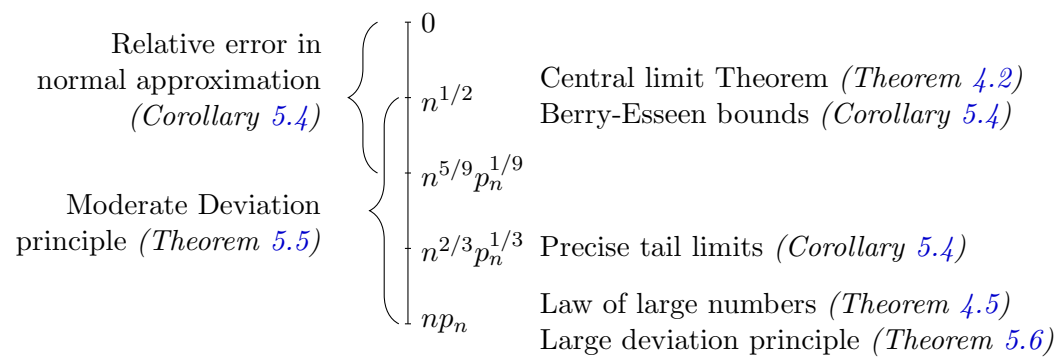
\begin{figure}[H]
        \small
        \begin{tikzpicture}[scale=.4]
            \usetikzlibrary{decorations.pathreplacing}
            \draw (0, 0) -- (0, 10);
            \draw [decorate,decoration={brace,amplitude=1em}] (-.3,0) -- (-.3, 7.5);
            \node at (-2.5, 3.25) [anchor=east, align=right]{Moderate Deviation \\ principle \emph{(Theorem~\ref{moderate_deviations})}};
            \draw [decorate,decoration={brace,amplitude=1em},] (-1,5) -- (-1, 10);
            \node at (-2.5, 8)  [anchor=east, align=right]{Relative error in \\  normal approximation \\ \emph{(Corollary~\ref{cormodphi})}};
            \draw (-.1, 10) -- (.1, 10) node[anchor=west]{0};
            \draw (-.1, 7.5) -- (.1, 7.5) node[anchor=west]{$n^{1/2}$};
            \node[anchor=west, align=left] at (4, 7.5) {Central limit Theorem \emph{(Theorem~\ref{convergence_hn})}  \\ Berry-Esseen bounds \emph{(Corollary~\ref{cormodphi})}};
            \draw (-.1, 5) -- (.1, 5) node[anchor=west]{$n^{5/9} p_n^{1/9}$};
            \draw (-.1, 2.5) -- (.1, 2.5) node[anchor=west]{$n^{2/3} p_n^{1/3}$};
            \node[anchor=west] at (4, 2.5) {Precise tail limits \emph{(Corollary~\ref{cormodphi})}};
            \draw (-.1, 0) -- (.1, 0) node[anchor=west]{$n p_n$};
            \node[anchor=west, align=left] at (4, 0) { Law of large numbers \emph{(Theorem~\ref{thm_lln})} \\ Large deviation principle \emph{(Theorem~\ref{large_deviation})}};
        \end{tikzpicture}
        \caption{ \textit{Overview of all the limiting results for the sequence $a_n^{-1} \big(H_n(s, t) - \e{H_n(s, t)}\big)$, depending on the order of $a_n$.}}
        \label{fig1}
    \end{figure}

    In the following we will first establish mod-$\phi$-convergence of the sequence   $(H_n (s,t))_{n\in \N}$ defined in (\ref{def_h}). Large deviation principles
    are discussed in the second part of this Section. As the moment generating function is closely related to cumulants, we first provide estimates for these objects 
in Lemma~\ref{lemma_cumulant1}  and   Lemma~\ref{lemma_cumulant2}, which are proved in Appendix~\ref{secc}. 

    \begin{lemma} \label{lemma_cumulant1}
    Let  $     r_{2n + 1, i, j} \sim \beta\left(\frac{p_n + 1}{2}(2n - i + 2) - j/2, j/2\right)$ denote independent beta distributed random variables, then  
the cumulants of the random variables 
        \begin{align*}
            S_n &= \sum \limits_{i = 1}^{\lfloor nt \rfloor - 1}\sum \limits_{j = 1}^{\lfloor p_ns \rfloor - 1}(\lfloor nt \rfloor - i)\log(r_{2n + 1, 2i, j}), \\
            S_n' &= \sum \limits_{i = 1}^{\lfloor nt \rfloor}\sum \limits_{j = 1}^{\lfloor p_ns \rfloor - 1}(\lfloor nt \rfloor - i + 1)\log(r_{2n + 1, 2i - 1, j})
        \end{align*}
        satisfy  the inequalities
        \begin{align*}
            \int_0^{t - 2/n} \Big (\frac{t - 2/n-x}{1 - x}\Big )^m \, dx &\le \frac{(-1)^{m}\kappa_m(S_n)}{n \frac{\lfloor p_ns -1 \rfloor\lfloor p_ns\rfloor}{(p_n + 1)^m} \frac{(m - 1)!}{4}} \le \Big(1 + \frac{m}{p_n}\Big ) \int_0^t \Big (\frac{t - x}{1 - x}\Big )^m \, dx \\
            \int_0^{t - 2/n} \Big (\frac{t - 1/n-x}{1 + 1/n - x}\Big )^m \, dx &\le \frac{(-1)^{m}\kappa_m^n(S_n')}{n \frac{\lfloor p_ns -1 \rfloor\lfloor p_ns\rfloor}{(p_n + 1)^m} \frac{(m - 1)!}{4}} \le \Big (1 + \frac{m}{p_n}\Big ) \Big (\int_0^t \Big (\frac{t - x}{1 - x}\Big )^m \, dx  + \frac{1}{n}\Big )
        \end{align*}
        for all $m \ge 1$. In particular
                \begin{align*}
            \max \{|\kappa_m(S_n)|, |\kappa_m(S_n')|\} \le 2 \cdot (m + 1)! n p_n^{2 - m}.
        \end{align*}
    \end{lemma}

    \begin{lemma}\label{lemma_cumulant2} 
        Let $p_{2n + 1, i, j} \sim \beta\left(\frac{p_n + 1}{2}(2n - i + 2) - j/2, \frac{p_n + 1}{2}(2n - i + 2)\right)$ denote independent beta distributed random variables, then the cumulants of the random variables
        \begin{align*}
            T_n &= \sum \limits_{i = 1}^{\lfloor nt \rfloor}\sum \limits_{j = 0}^{\lfloor p_ns \rfloor - 1} \log(p_{2n + 1, 2i, j}^{\lfloor nt \rfloor - i + 1}(1 - p_{2n + 1, 2i, j})^{\lfloor nt \rfloor - i}), \\
            T_n' &= \sum \limits_{i = 1}^{\lfloor nt \rfloor}\sum \limits_{j = 0}^{\lfloor p_ns \rfloor - 1} \log(p_{2n + 1, 2i - 1, j}^{\lfloor nt \rfloor - i + 1}(1 - p_{2n + 1, 2i - 1, j})^{\lfloor nt \rfloor - i + 1})
        \end{align*}
        satisfy the inequalities
        \begin{align*}
            \max\{|\kappa_m(T_n)|, |\kappa_m(T_n')|\} &\le 6 \cdot 4^m (m + 1)! p_n^{-m} \left(n p_n + (\log(n) + 1)p_n^2\right) \\
            &\le 12 \cdot 4^m (m + 1)! n p_n^{2-m}
        \end{align*}
        for all $m \ge 1$.
    \end{lemma}

    \begin{thm}[Mod-Gaussian-convergence]\label{mod_phi_conv}
        For any fixed $s, t \in (0, 1]$ the sequence 
        \begin{align*}
            \left(\frac{p_n}{n}\right)^{1/3}(H_n(s, t) - \e{H_n(s, t)})
        \end{align*}
        converges mod-Gaussian on any strip $S_{(a, b)}$ ($-\infty < a <  b < \infty $)
        with speed 
            \begin{align}
            t_n = \left(\frac{p_n}{n}\right)^{2/3} \kappa_2(H_n(s, t))
            \sim{}& n^{1/3} p_n^{2/3} \frac{s^2}{2} \int_0^t \Big (\frac{t - x}{1 - x}\Big )^2 \, dx \to \infty \label{speed_tn}
        \end{align}     
       and limiting function
        \begin{align*}
            \psi(z) = \exp\Big (-z^3 \frac{s^2}{6} \int_0^t \Big (\frac{t - x}{1 - x}\Big )^3 \, dx\Big).
        \end{align*}
    \end{thm}
    \begin{proof}
      Recalling the definition of $   H_n(s, t)$ in   (\ref{def_h}), it follows by an application of
     Lemma \ref{lemma_cumulant1} and \ref{lemma_cumulant2}    that $t_n \to \infty$. More precisely
        \begin{align*}
            t_n ={}& 2\frac{p_n^{2/3}}{n^{2/3}}  n \frac{(\lfloor p_n s\rfloor - 1) \lfloor p_n s\rfloor}{(p_n + 1)^2} \frac{1}{4} \Big (\int_0^t \Big (\frac{t - x}{1 - x}\Big)^2 \, dx + o(1)\Big ) + \frac{p_n^{2/3}}{n^{2/3}} O(n p_n^{-1} + \log(n))\nonumber \\
            \sim{}& n^{1/3} p_n^{2/3} \frac{s^2}{2} \int_0^t \Big (\frac{t - x}{1 - x}\Big )^2 \, dx \to \infty ,
        \end{align*}    
        which proves \eqref{speed_tn}.
        By the definition of the cumulant-generating function the equation
        \begin{align*}
            &\log \Big\{\exp(-t_n z^2/2) \e{\exp(z p_n^{1/3} n^{-1/3} (H_n(s, t)) - \e{H_n(s, t)})}\Big\} \\
            ={}& \sum \limits_{j = 3}^\infty \frac{z^j}{j!} \kappa_j(p_n^{1/3} n^{-1/3}H_n(s, t)) \\
            ={}& \frac{z^3}{6} p_n n^{-1} \kappa_3(H_n(s, t)) + \sum \limits_{j = 4}^\infty \frac{z^j}{j!} p_n^{j/3} n^{-j/3} \kappa_j(H_n(s, t)).
        \end{align*}
        holds. As before we get by an application of Lemma \ref{lemma_cumulant1} and \ref{lemma_cumulant2}  
        \begin{align*}
            \frac{z^3}{6} p_n n^{-1} \kappa_3(H_n(s, t)) 
            ={}& - \frac{z^3}{3} \Big (\int_0^t \Big (\frac{t - x}{1 - x}\Big)^3 \, dx + o(1)\Big ) \frac{p_n \lfloor p_n s -1 \rfloor \lfloor p_n s\rfloor}{(p_n + 1)^3} \frac{1}{2} \\
            & + z^3 O(p_n n^{-1} p_n^{-3} (n p_n + (\log(n) + 1) p_n^2)) \\
            ={}&-\frac{z^3}{6} s^2 \int_0^t \Big (\frac{t - x}{1 - x}\Big)^3 \, dx + o(1)
        \end{align*}   
The remainder converges locally uniform to $0$, which follows from the inequality  
        \begin{align*}
            \Big |\sum \limits_{j = 4}^\infty \frac{z^j}{j!} p_n^{j/3} n^{-j/3} \kappa_j(H_n(s, t))\Big | 
            \le{}& 24 \sum \limits_{j = 4}^\infty \frac{|z|^j}{j!} p_n^{j/3}n^{-j/3} (j + 1)! 4^j n p_n^{2 - j} \\
            ={}& 9216 |z|^4 n^{-1/3} p_n^{-2/3}\sum \limits_{j = 0}^\infty (j + 4) \Big (\frac{4|z|}{n^{1/3} p_n^{2/3}}\Big )^j ~. 
        \end{align*}
    \end{proof}
    From the mod-$\phi$-converge established in Theorem \ref{mod_phi_conv}  we can derive several conclusions. Exemplarily 
    we mention the following results, which are obtained by an application of   Proposition 4.1.1, Theorem 4.3.1 and  Theorem 4.2.1 in \cite{femeni2016}
    \begin{cor}  \label{cormodphi}~~\\ 
    \vspace{-.5cm}
    \begin{itemize}
    \item[(1)]  Berry-Esseen bound
        \begin{align*}
            \sup \limits_{x \in \mathbb{R}} \Big|\pr \Big ( {\frac{H_n(s, t) - \e{H_n(s, t)}}{\sqrt{\var{H_n(s, t)}}} \le x}  \Big ) - \Phi(x)\Big| = o\left(\frac{1}{\sqrt{t_n}}\right) = o\left(n^{-1/6} p_n^{-1/3}\right).
        \end{align*}
    \item[(2)] Relative error in normal approximation: 
        For any sequence $x_n$ satisfying $x_n = o(t_n^{1/6}) = o(n^{1/18} p_n^{1/9})$ we have
        \begin{align*}
            \pr \Big ( {\frac{H_n(s, t) - \e{H_n(s, t)}}{\sqrt{\var{H_n(s, t)}}} \le x_n}  \Big) = \Phi(x_n)(1 + o(1)).
        \end{align*}
      \item[(3)]  Precise tail limits: 
        For $x > 0$ we have
        \begin{align*}
            &\pr \Big ( {H_n(s, t) - \e{H_n(s, t)} \ge \frac{n^{1/3} t_n}{p_n^{1/3}} x } \Big)  \\
            ={}& \frac{1}{x \sqrt{2\pi t_n}} \exp\Big (-t_n x^2/2 -x^3 \frac{s^2}{6} \int_0^t \Big (\frac{t - y}{1 - y}\Big )^3 \, dy\Big)(1 + o(1)).
        \end{align*}
For $x < 0$ we have
        \begin{align*}
            &\pr \Big ( {H_n(s, t) - \e{H_n(s, t)} \le \frac{n^{1/3} t_n}{p_n^{1/3}} x } \big )  \\
            ={}& \frac{1}{-x \sqrt{2\pi t_n}} \exp\Big (-t_n x^2/2 -x^3 \frac{s^2}{6} \int_0^t \Big (\frac{t - y}{1 - y}\Big )^3 \, dy\Big )(1 + o(1)).
        \end{align*}
    \end{itemize}
    \end{cor}
     In  Part (3)  of Corollary \ref{cormodphi}, the speed of the large deviation bound is $t_n \sim C n^{1/3} p_n^{2/3}$ (c.f. formula (\ref{speed_tn})) and $H_n(s, t)$ is rescaled by $\frac{n^{1/3} t_n}{p_n^{1/3}} \sim C n^{2/3} p_n^{1/3}$. This rescaling correponds to  a moderate deviation principle. Using the Gärtner-Ellis theorem we can strengthen the large deviation principle to a speed of $n p_n^2$ and a rescaling with $n p_n$ [see Theorem \ref{large_deviation} below]. Before we do this,
     we will provide a more general version of a moderate deviation principle for the sequence $H_n(s, t)$. 
          This result contains part (3) of Corollary \ref{cormodphi} as a 
           special case (using  $a_n = n^{2/3} p_n^{1/3}$). However, the  latter provides  more precise information about the limit.
    
    \begin{thm}[Moderate deviations] \label{moderate_deviations}
        Let $s, t \in (0, 1]$ be fixed and let $a_n$ be a sequence with  $\sqrt{n} = o(a_n)$ and $a_n = o(n p_n)$. Then
        \begin{align*}
            X_n = \frac{H_n(s, t) - \e{H_n(s, t)}}{a_n}
        \end{align*}
        satisfies a large deviation principle with speed $b_n = \frac{a_n^2}{n}$ and good rate function
        \begin{align*}
            I(x) = x^2 \Big (s^2 \int_0^t \Big (\frac{t - y}{1 - y}\Big )^2 \, dy\Big)^{-1}
        \end{align*}
    \end{thm}
    \begin{proof}
        We will use  the Gärtner-Ellis theorem [see Theorem 2.3.6 in \cite{demzeit2009} and  the subsequent remark].    From Lemma \ref{lemma_cumulant1} and Lemma \ref{lemma_cumulant2} it follows 
        \begin{align*}
            \frac{1}{b_n}\log \e{\exp(b_n x X_n/a_n)} =& \sum \limits_{j = 2}^\infty \frac{x^j b_n^{j - 1}\kappa_j(X_n)}{j!a_n^j} 
            ={} \frac{x^2}{2n} \kappa_2(X_j) + O\Big (\sum \limits_{j = 3}^\infty \frac{x^j a_n^{j - 2}}{j!} (j + 1)! \frac{n}{p_n^{j-2 }} \Big ) \\
            ={}& \frac{x^2}{2n} \Big (\int_0^t \Big (\frac{t - y}{1 - y}\Big )^2 \, dy + o(1)\Big ) n \frac{(\lfloor p_n s \rfloor - 1) \lfloor p_n s \rfloor}{(p_n + 1)^2} \frac{1}{2} \\
            &+ \frac{x^2}{2n} O\Big (\frac{n}{p_n} + (\log(n) + 1) \Big )
            + O\Big (\sum \limits_{j = 3}^\infty x^j (j + 1) \Big (\frac{a_n}{n p_n}\Big )^{j - 2}\Big) \\
            ={}& x^2 \frac{s^2}{4} \int_0^t \Big (\frac{t - y}{1 - y}\Big )^2 \, dy + o(1).
        \end{align*}
        Since $I(x)$ is the Fenchel-Legendre transform of the limit on the right hand side of the last equation,      
                the Gärtner-Ellis theorem yields the desired result.
    \end{proof}

    The large deviation principle in Theorem \ref{moderate_deviations} is called a moderate deviation principle, because the scale at which the deviations occur is between the scaling of a central limit theorem and the scale of a law of large numbers. Indeed, at scale $\sqrt{n}$ the sequence $H_n(s, t) - \e{H_n(s, t)}$ is asymptotically Gaussian by Theorem \ref{convergence_hn} and at scale $n p_n$ the sequence $H_n(s, t)$ satisfies a law of large numbers by Theorem~\ref{thm_lln}.

    \begin{thm}[Large deviations]\label{large_deviation}
      The sequence $\frac{H_n(s, t)}{n p_n}$ satisfies a large deviation principle with speed $n p_n^2$ and good rate function
        \begin{align*}
            \Lambda^*(x) = \sup\limits_{\lambda > -\frac{1}{t}} \Big (\lambda x + \frac{s^2}{2} \int_0^t \log\Big  (1 + \lambda \frac{t - y}{1 - y}\Big  ) \, dy\Big ).
        \end{align*}
      In particular   the sequence $\frac{H_n(s, 1)}{n p_n}$ satisfies a large deviation principle with speed $n p_n^2$ and good rate function
        \begin{align*}
            \Lambda^*(x) = 
            \begin{cases}
                -(x + s^2/2) + s^2 \log(s) + \frac{s^2}{2} \log(-2x) & x < 0 \\
                \infty & x \ge 0
            \end{cases}.
        \end{align*}
    \end{thm}
    \begin{proof}
      The Gärtner-Ellis theorem, Lemma \ref{lemma_cumulant1} and Lemma~\ref{lemma_cumulant2} yield
        \begin{align*}
            \Lambda(x) = & \limsup \limits_{n \to \infty} \frac{1}{n p_n^2} \log \e{\exp(p_n x H_n(s, t))} 
            ={} \limsup \limits_{n \to \infty} \sum \limits_{j = 1}^\infty \frac{x^j}{j!} p_n^{j - 2} n^{-1} \kappa_j(H_n(s, t)) \\
            ={}& \limsup \limits_{n \to \infty} \sum \limits_{j = 1}^\infty\Big\{ \frac{1}{2} \frac{(-x)^j}{j} (s^2 + o(1)) \Big  (\int_0^t \Big  (\frac{t - y}{1 - y}\Big  )^j \, dy + o(1)\Big ) \\
            &\qquad \qquad +O\Big  (4^j x^j (j + 1) \Big  (\frac{1}{p_n} + \frac{\log(n) + 1}{n}\Big  )\Big  )\Big\} \\
            ={}& \frac{s^2}{2} \int_0^t \sum \limits_{j = 1}^\infty \Big (-x \frac{t - y}{1 - y}\Big  )^j \frac{1}{j} \, dy 
            ={} -\frac{s^2}{2} \int_0^t \log\Big  (1 + x \frac{t - y}{1 - y}\Big  ) \, dy.
        \end{align*}
        The remaining part follows by a straightforward calculation of 
      \begin{align*}
            \Lambda^*(x) = \sup \limits_{\lambda > -1}  \Big\{ \lambda x + \frac{s^2}{2} \log(1 + \lambda) \Big\}.
        \end{align*}
         \end{proof}

    \begin{rem}   {\rm 
    Note that   the mod-Gaussian convergence provides  moderate deviation principles  while the application of the  Gärtner-Ellis theorem yields the full large deviation principle. 
    This is most likely due to the fact that we used a rather simple approximation when proving Theorem \ref{mod_phi_conv}, based on  upper bounds for the cumulants
    of order larger than three.  In contrast to this, Theorem \ref{large_deviation} uses all the cumulants in calculating the limiting function. 
        This situation is similar to Example 2.1.3 in \cite{femeni2016}, where  a sequence is proven to converge mod-Poisson as well as mod-Gaussian, and
         the results obtained from the mod-Poisson convergence are much stronger than the ones obtained from mod-Gaussian convergence. \\
        For $t = 1$ the limit of the cumulant generating function $\Lambda(x)$ in Theorem \ref{large_deviation} is $\Lambda(x) = -\frac{s^2}{2} \log(1 + x)$. This is the cumulant-generating function of a negative $\gamma(\tfrac{s^2}{2}, 1)$-distribution. A comparison with Theorem 4.2.1 from \cite{femeni2016} suggests the conjecture 
         that $C \cdot (p_n + O(1))H_n(s, t)$ converges mod-$\gamma$ with a speed $t_n \sim C n p_n^2$, as this would imply the large deviation principle in Theorem \ref{large_deviation}.
        }
    \end{rem}

\bigskip

{\bf Acknowledgements.} 
%The authors would like to thank
% M. Stein who typed this manuscript with considerable technical expertise.
The work of H. Dette and D. Tomecki  was partially supported by the Deutsche
Forschungsgemeinschaft (DFG Research Unit 1735, DE 502/26-2, RTG 2131). The authors would like to thank Christoph Thäle for some helpful discussion about mod-$\phi$-convergence.

        % Bibliography
    \addcontentsline{toc}{section}{References}
    \bibliography{sources}

    % Appendix
\appendix

\section{Technical Details } 
\subsection{Proof of Theorem \ref{limit_thm}} \label{sec6}

  We begin with some general remarks regarding the assumptions made in Theorem \ref{limit_thm}.
Conditions \ref{cond_indep}, \ref{cond_center}, \ref{cond_var_bound} and \ref{cond_var_conv} are 
rather common and  necessary to apply the Lindeberg central limit theorem.
     \ref{cond_moments} is used to prove the Lindeberg condition and in combination with \ref{cond_cov} it
     also allows us to bound the increments of the process via the Markov inequality. The latter condition also guarantees that a continuous 
     version of the limiting process exists.  Condition \ref{cond_lattice} essentially means that
     the process 
      $Z_n$ is constant on all rectangular sets of the form $$\prod \limits_{j = 1}^k [a_j h_n^{(j)}, (a_j + 1)h_n^{(j)}),$$
       where $a_j$ are integers satisfying $0 \le a_j \le \frac{1}{h_n^{(j)}}$.

       The proof is split into several parts. First, we prove in Lemma \ref{fin_conv} that the finite-dimensional distributions of 
       $Z_n$ converge weakly to a centered Gaussian distribution with covariance kernel $f$.   This also implies that $f$ is 
       nonnegative definite, and therefore we can conclude that a centered Gaussian process with covariance kernel $f$ exists. 
       We then show in Theorem \ref{cont_proc} that this process can be chosen to have continuous paths.     
       Finally, we show in Theorem \ref{as_equi} that the process $Z_n$ is asymptotically tight. The assertion now 
       follows from Theorem 1.5.4 in \cite{vanwell1996}.

    \begin{lemma}[Convergence of finite-dimensional distributions] \label{fin_conv}
        For any $t_1, \ldots, t_m \in [0, 1]^k$ we have
        \begin{align*}
        Z_n^\star =     (Z_n(t_1), \ldots, Z_n(t_m)) \xrightarrow{d} \mathcal{N}(0, \Sigma),
        \end{align*}
        where the elements of $\Sigma \in \mathbb{R}^{m \times m}$ are given by
$
            \Sigma_{i, j} = f(t_i, t_j).
$
    \end{lemma}
    \begin{proof}
        By the Cram\'{e}r-Wold theorem it is sufficient to prove the weak convergence
        \begin{align*}
            c^t Z_n^\star \xrightarrow{d} \mathcal{N}(0, c^t \Sigma c)
        \end{align*}
        for any arbitrary vector $c = (c_1, \ldots, c_m)^t \in \mathbb{R}^m$. For this purpose we use the Lindeberg central limit theorem in the  the form of Theorem 5.12 of \cite{kall2002}. The first step is showing that
        \begin{align*}
            \var{c^t Z_n^\star} \to c^t \Sigma c,
        \end{align*}
        which is equivalent to proving $\cov{Z_n(s)}{Z_n(t)} \to f(s, t)$ for all $s, t \in [0, 1]^k$. This follows directly from assumptions \ref{cond_indep} and \ref{cond_var_conv} and the following calculation
        \begin{align*}
            \cov{Z_n(s)}{Z_n(t)} &= \sum \limits_{i, j \in T_n} g_n(i, s) g_n(j, t) \cov{X_n(i)}{X_n(j)} \\
            &=\sum \limits_{i \in T_n} g_n(i, s) g_n(i, t) \var{X_n(i)} \to f(s, t).
        \end{align*}   
        To prove the Lindeberg-condition, observe the inequality
        \begin{align*}
            \e{X^2 I\{|X| > \varepsilon\}} \le \e{\frac{X^4}{\varepsilon^2} I\{|X| > \varepsilon\}} \le \e{X^4}\varepsilon^{-2},
        \end{align*}
       and note that        \begin{align*}
            \e{X_n^4(i)} \le \e{X_n^{2k + 2}(i)}^{\frac{2}{k + 1}} \le C^{\frac{2}{k + 1}} \e{X_n^2(i)}^2 \le C \e{X_n^2(i)}^2
        \end{align*}
         by assumption \ref{cond_moments} and Jensen's inequality.
        Define $ c^\star = \max \{|c_1|, \ldots, |c_m|\}$
        and observe that
        \begin{align*}
            c^tZ_n^\star = \sum \limits_{j = 1}^m c_j Z_n(t_j) = \sum \limits_{i \in T_n} \Big (\sum \limits_{j = 1}^m c_j g_n(i, t_j)\Big ) X_n(i).
        \end{align*}
 A further application of Jensen's inequality and  \ref{cond_var_bound}  yield
        \begin{align*}
            &\sum \limits_{i \in T_n} \Big (\sum \limits_{j = 1}^m c_j g_n(i, t_j)\Big)^4 \e{X_n^4(i)} \le \sum \limits_{i \in T_n} m^3 \sum \limits_{j = 1}^m c_j^4 g_n^4(i, t_j) \e{X_n^4(i)}\\
            \le{}&C m^3(c^\star)^4\sum \limits_{j = 1}^m\Big\{\Big(\sup \limits_{i \in T_n} g_n^2(i, t_j)\var{X_n(i)}\Big) \Big(\sum \limits_{i \in T_n} 
            g_n^2(i, t_j) \var{X_n(i)}\Big)\Big\} \\
            \le{}& Cm^3(c^\star)^4 \sum \limits_{j = 1}^m \Big(\sup \limits_{i \in T_n} g_n^2(i, t_j)\var{X_n(i)}\Big) \left(f(t_j, t_j)+o(1)\right) \xrightarrow{n \to \infty} 0,
        \end{align*}
        which proves the Lindeberg-condition.
    \end{proof}
    
    \begin{thm}[Continuity of the limit process] \label{cont_proc}
        There exists a continuous, centered Gaussian process with covariance kernel $f$.
    \end{thm}
    \begin{proof}
        Let $G$ be a centered Gaussian process with covariance function $f$. For arbitrary vectors $s = (s_1, \ldots, s_k), t = (t_1, \ldots, t_k) \in [0, 1]^k$ and $1 \le i \le k + 1$ define
        \begin{align*}
            M_i = (s_1, \ldots, s_{i - 1}, t_i, \ldots, t_k).
        \end{align*}
        From assumptions \ref{cond_var_conv} and \ref{cond_cov} we can conclude
        \begin{align*}
            \var{G(M_i) - G(M_{i + 1})} = \lim \limits_{n \to \infty} \var{Z_n(M_i) - Z_n(M_{i + 1})} \le C|t_i - s_i|.
        \end{align*}
        This yields%  by an application of Jensen's inequality
        \begin{align*}
            \e{(G(t) - G(s))^{2k + 2}}
     %       &={}k^{2k + 2}\e{\left(\frac{1}{k}\sum \limits_{i = 1}^k\left\{G(M_i) - G(M_{i + 1})\right\}\right)^{2k + 2}} \\
            &\le{} k^{2k + 1} \sum \limits_{i = 1}^k \e{(G(M_i) - G(M_{i + 1}))^{2k + 2}} \\
            &={} k^{2k + 1} \sum \limits_{i = 1}^k (2k + 1)!! \var{G(M_i) - G(M_{i + 1})}^{k + 1}\\
            &\le{} k^{2k + 1} \sum \limits_{i = 1}^k (2k + 1)!!(C|t_i - s_i|)^{k + 1} \\
            &\le  (Ck^2)^{k + 1} (2k + 1)!! \cdot \|t - s\|_\infty^{k + 1},
        \end{align*}
     and  Theorem 3.23 in \cite{kall2002} implies the existence of  a continuous version of the process $G$ .
    \end{proof}
The proof of asymptotic tightness of $Z_n$ requires some preparations. Typically, the asymptotic tightness of a one-dimensional random process $H_n \in l^{\infty}([0, 1])$ is proven by showing a bound of the form
    \begin{align*}
        \e{(H_n(s) - H_n(t))^b} \le C|t-s|^a,
    \end{align*}
    where $a > 1$ and $b > 0$ are some parameters. Theorem 6 on p. 51 in \cite{shorwell2009} then yields the asymptotic tighness of $H_n$. However, since we are mostly interested in partial sum processes, such an inequality cannot hold. This is due to the discontinuity of partial sum processes at fixed points $\frac{1}{n}, \frac{2}{n}, \ldots, \frac{n}{n}$. We will therefore use a similar, but slightly more delicate argument generalizing  Theorem 6.2 in \cite{bill1971} to 
    more than one dimension.   Informally speaking, we  show 
     that an increment $|Z(t) - Z(s)|$ is small with high probability (w.h.p.), if $t$ and $s$ are close to each other and only differ in one coordinate. From this we deduce that the increments of $Z$ are ``simultaneously'' small in the sense that $\sup \limits_{s, t}|Z(t) - Z(s)|$ is small w.h.p..
    
   In order to achieve this, we use a chaining type argument and 
  define  a ``dyadic lattice'' on the  cube $[0, 1]^k$. Starting with  the $2^k$ vertices  $ \big  \{(i_1,\ldots ,i_k) |~ i_1,\ldots ,i_k \in \{0, 1\}  \big \}$ 
  we subdivide the lattice in each step, to gain the lattice  $ \big  \{(i_1,\ldots ,i_k) /2^{n} |~ i_1,\ldots ,i_k \in \{0, 1,  \ldots , 2^n\}  \big \}$ 
after $n$ steps. 
  Then an induction argument shows that the increments to the nearest neighbors \emph{within} a lattice of length $2^{-n}$ are small w.h.p..
      Using the assumptions on the increments of $Z$, we can find for all points $s, t$ in the $n$-th lattice points $S, T$ in the $(n - 1)$-th lattice, so that $|Z(s) - Z(S)|$ and $|Z(t) - Z(T)|$ are small w.h.p..  Summing up all  ``small increments'', we can see that all increments $|Z(s) - Z(t)|$ are ``simultaneously small'' w.h.p. for all dyadic rationals $s, t$. Using the right-continuity of $Z$ this can be strengthened to hold for all real numbers $s, t$.
    
    This argument is visualized in the following two graphics for the two-dimensional ($k = 2$) case.
    \begin{figure}[H]
        \newcommand{\tikzrescale}{.7}
        \newcommand{\arrowwidth}{.7}
        \centering
        \begin{minipage}{.48\textwidth}
            \centering
            \begin{tikzpicture}[scale=\tikzrescale]
                \draw[lightgray] (-.2,-.2) grid [xstep=1,ystep=1] (8.2,8.2);
                \draw (-.2,-.2) grid [xstep=2,ystep=2] (8.2,8.2);
                
                \draw[fill=gray] (2,3) circle (.12) node[anchor=south east] {$s$};
                \draw (2,2) circle (.1);
                \draw (2,4) circle (.1);
                
                \draw[fill=black] (5,7) circle (.12) node[anchor=south east] {$t$};
                \draw (4, 6) circle (.1);
                \draw (4, 8) circle (.1);
                \draw (6, 6) circle (.1);
                \draw (6, 8) circle (.1);
                
                \draw[dashed,line width=1, <->] (2.15, 3.15) -- (4.85, 6.85);
            \end{tikzpicture}
            \caption{\it Points $s$ and $t$ on the $n$-th lattice with feasible points $S$, $T$ on the $(n - 1)$-th lattice marked.}
        \end{minipage}\hfill
        \begin{minipage}{.48\textwidth}
            \centering
            \begin{tikzpicture}[scale=\tikzrescale]
                \draw[lightgray] (-.2,-.2) grid [xstep=2,ystep=2] (8.2,8.2);
                \draw (-.2,-.2) grid [xstep=4,ystep=4] (8.2,8.2);
                
                \draw[line width=\arrowwidth,->] (2, 3) .. controls (1.5, 2.8) and (1.5, 2.3) .. (1.9, 2.1);
                \draw[fill=gray] (2,3) circle (.12) node[anchor=south east]{$s$};
                \draw[fill=gray] (2,2) circle (.12) node[anchor=north east]{$S$};
                \draw (0,0) circle (.1);
                \draw (0,4) circle (.1);
                \draw (4,0) circle (.1);
                \draw (4,4) circle (.1);
                
                \draw[line width=\arrowwidth,->] (5,7) .. controls (5.7, 7.3) and (6.3, 6.7) .. (6.05, 6.15);
                \draw[fill=black] (5,7) circle (.12) node[anchor=south east]{$t$};
                \draw[fill=black] (6,6) circle (.12) node[anchor=north west]{$T$};
                \draw (4,8) circle (.1);
                \draw (8,8) circle (.1);
                \draw (8,4) circle (.1);
                
                \draw[dashed,line width=1, <->] (2.2, 2.2) -- (5.8, 5.8);
            \end{tikzpicture}
            \caption{\it Feasible points $S$ and $T$ were chosen. New feasible points for $S$ and $T$ on the \hbox{$(n - 2)$-th} lattice are marked.}
        \end{minipage}
    \end{figure}

      While the assumptions in the following theorem may seem technical, they can typically be proven
      by a simple application of  the Markov-inequality and estimates on the moments of $Z$.
    \begin{thm}[Global increments] \label{sup_ineq}
        Let $Z: [0, 1]^k \to \mathbb{R}$ be a process, which is right-continuous in every coordinate. Define for $1 \le i \le k$ and $s = (s_1, \ldots, s_k)$ the increment in the $i$-th coordinate by:
        \begin{align*}
            m^i(s_1, \ldots, s_k, r, t) = \min\{&|Z(s_1, \ldots, s_{i - 1}, t, s_{i + 1}, \ldots, s_k) - Z(s)|, \\
            &|Z(s_1, \ldots, s_{i - 1}, r, s_{i + 1}, \ldots, s_k) - Z(s)|\}.
        \end{align*}
        Assume that there exist constants $\gamma > 0$, $\delta > k$ such that for all $1 \le i \le k$ and $r \le s_i \le t$ the inequality
        \begin{align}
            \p{m^i(s_1, \ldots, s_k, r, t) > \lambda} \le C \lambda^{-\gamma} |t - r|^\delta \label{min_increment}
        \end{align}
        holds with a universal constant $C$. Further assume that there exists a function $\eta$ such that the inequality
        \begin{align}
              \pr \big(    {|Z(s) - Z(t)| > \varepsilon}  \big) \le \eta(\|t - s\|_\infty, \varepsilon) \label{inc_border}
        \end{align}
        is satisfied. Then we have 
        \begin{align*}
             \pr \Big(  {\sup\limits_{s, t \in [0, 1]^k} |Z(s) - Z(t)| > 4k\lambda} \Big)  \le C C' \lambda^{-\gamma} + 4^k\eta(1, \lambda),
        \end{align*}
        where $C'$ is a universal constant that only depends on $\gamma$, $\delta$ and $k$.
    \end{thm}
    \begin{proof}
        Let $\theta_0, \theta_1, \theta_2, \ldots$ be arbitrary positive numbers and consider the event
        \begin{align*}
            M =  \bigcap \limits_{m = 1}^\infty \bigcap \limits_{i_1,\ldots, i_k = 1}^{2^m - 1} \bigcap \limits_{j = 1}^k \Big  \{m^j\Big  (\frac{i}{2^m}, \frac{i_j - 1}{2^m}, \frac{i_j + 1}{2^m}\Big) \le \lambda \theta_m\Big \} 
            \bigcap \Big \{\max \limits_{s, t \in \{0, 1\}^k} |Z(t) - Z(s)| \le \lambda \theta_0 \Big  \}.
        \end{align*}
        By assumptions (\ref{min_increment}) and (\ref{inc_border}) the complimentary event has a probability of at most
        \begin{align*}
            C \lambda^{-\gamma} k 2^{\delta} \sum \limits_{m = 1}^ \infty 2^{m(k - \delta)} \theta_m^{-\gamma} + 4^k \eta(1, \lambda \theta_0).
        \end{align*}
        On $M$ the following inequality 
        \begin{align}
            \Big |Z\Big(\frac{s}{2^m}\Big) - Z\Big (\frac{t}{2^m}\Big)\Big | \le 2k \sum \limits_{i = 0}^m \lambda \theta_i \qquad \label{ineq_incr}
        \end{align}
        holds for all $s, t \in \{0, \ldots, 2^m\}^k$, which is obtained by  an induction argument with respect to $m$. 
        If the inequality holds for $m - 1$ we successively choose $s_1', \ldots, s_k'$ as follows.
        Assume we have already chosen $s_1', \ldots, s_j'$ and set $S_j = \big(\frac{s_1'}{2^{m - 1}}, \ldots, \frac{s_j'}{2^{m - 1}}, \frac{s_{j + 1}}{2^m}, \ldots, \frac{s_k}{2^m}\big)^t$. Then
        \begin{align*}
            s_{j + 1}' = 
            \begin{cases}
                \frac{s_{j + 1}}{2} &\text{if } s_{j + 1} \text{ is even} \\
                \operatorname{argmin} \limits_{0 \le p \le 2^{m - 1}} \left| Z\left(\frac{s_1}{2^{m - 1}}, \ldots, \frac{s_{i - 1}}{2^{m - 1}}, \frac{p}{2^{m - 1}}, \frac{s_{i + 1}}{2^m}, \ldots, \frac{s_k}{2^m}\right) - Z(S_j)\right| & \text{if } s_{j + 1} \text{ is odd}
            \end{cases}.
        \end{align*}
        The sequence $S_j$ has the following three properties
        \begin{enumerate}
            \item $S_0 = \frac{s}{2^m}$.
            \item $|Z(S_j) - Z(S_{j + 1})| < \lambda \theta_m$ for $0 \le j < k$.
            \item $S_k \in \big \{ (i_1, \ldots , i_k)/2^{m - 1} ~|~i_1, \ldots , i_k \in    \{0, \ldots, 2^{m - 1}\} \big\}$.
        \end{enumerate}    
        a similar constructions yield a sequence $T_j$ from $t_j$, which proves (\ref{ineq_incr}):
        \begin{align*}
            \Big |Z\Big(\frac{s}{2^m}\Big) - Z\Big (\frac{t}{2^m}\Big)\Big| 
            %  = \left|Z(S_k) - Z(T_k) + \sum \limits_{j = 0}^{k - 1} \big (Z(S_j) - Z(S_{j + 1}) + Z(T_j) - Z(T_{j + 1})\big )\right| \\
            &\le  |Z(S_k) - Z(T_k)| + \sum \limits_{j = 0}^{k - 1} \big (|Z(S_j) - Z(S_{j + 1})| + |Z(T_j) - Z(T_{j + 1})|\big) \\
           &  \le \Big (2k\sum \limits_{i = 0}^{m - 1} \lambda \theta_i\Big ) + 2k \lambda \theta_m = 2k\sum \limits_{i = 0}^m \lambda \theta_i.
        \end{align*}
        We now choose $\theta_i = (i + 1)^{-2}$. Then the inequality
        \begin{align*}
            |Z(s) - Z(t)| \le 2k \lambda\sum \limits_{i = 0}^\infty (i + 1)^{-2} = k\lambda \frac{\pi^2}{3} \le 4k\lambda
        \end{align*}
        holds on $M$ for all dyadic rational points $s, t \in [0, 1]^k$. Since the paths of $Z$ are right-continuous in every coordinate, the inequality is 
        also satisfied for all real vectors $s, t \in [0, 1]^k$. The theorem now follows by choosing
        \begin{align*}
            C' = k 2^{\delta} \sum \limits_{m = 1}^ \infty 2^{m(k - \delta)} \theta_m^{-\gamma} = k 2^{\delta} \sum \limits_{m = 1}^ \infty 2^{m(k - \delta)} (m + 1)^{-\gamma}.
        \end{align*}
    \end{proof}
    
    \begin{cor}[local increments]\label{small_inc}
        Assume the process $Z$ satisfies the assumptions of Theorem \ref{sup_ineq}. Then for all $\varepsilon > 0$ and $r \in [0, 1]^k$
        \begin{align*}
            \pr \Big ( ~{\sup \limits_{\substack{\|s\|_\infty, \|t\|_\infty < \varepsilon \\ r+s, r + t \in [0, 1]^k}} |Z(r + s) - Z(r + t)| > 4k \lambda}
            \Big)  \le \varepsilon^\delta C C' \lambda^{-\gamma} + 4^k\eta(\varepsilon, \lambda).
        \end{align*}
    \end{cor}
    \begin{proof}
        Consider the modified process $Z': [0, 1]^k \to \mathbb{R}, Z'(s) = Z((r + \varepsilon s)\wedge 1)$. This process satisfies the requirements of Theorem \ref{sup_ineq}, if $\eta$ is replaced by $\eta(\epsilon \cdot, \cdot)$ and $C$ is eplaced by $C \varepsilon^\delta$.
    \end{proof}

    \begin{lemma}[a Rosenthal-type inequality] \label{rosenthal}
        Let $X_1, \ldots, X_n$ be independent centered random variables and $p > 2$. Then the inequality 
        \begin{align*}
            \ex \Big [~{  \Big  |\sum \limits_{i = 1}^n X_i  \Big  |^p}   \Big ] \le R(p)  \Big  (\sum \limits_{i = 1}^n \e{|X_i|^p}^{2/p} \Big  )^{p/2}
        \end{align*}
        holds, where $R(p)$ is a universal constant only depending on $p$.
    \end{lemma}
    \begin{proof}
        By Jensen's inequality the inequalities
        \begin{align*}
            \sum \limits_{i = 1}^n \ex  \big  [ {\left|X_i\right|^p}  \big ]  
&=  \Big  (n^{2/p} \Big  (\sum \limits_{i = 1}^n \frac{1}{n}\e{|X_i|^p} \Big  )^{2/p} \Big )^{p/2} 
             \le  \Big (n^{2/p - 1} \sum \limits_{i = 1}^n \e{|X_i|^p}^{2/p} \Big  )^{p/2} \\
            &\le  \Big  (\sum \limits_{i = 1}^n \e{\left|X_i\right|^p}^{2/p} \Big  )^{p/2} \\ 
          \Big  (\sum \limits_{i = 1}^n \e{X_i^2} \Big  )^{p/2} & \le  \Big  (\sum \limits_{i = 1}^n \e{|X_i|^p}^{2/p} \Big  )^{p/2}
        \end{align*}
        hold. The assertion now follows from Rosenthal's inequality [c.f. Theorem 3 in \cite{ros1970}]:
        \begin{align*}
            \ex  \Big  [ { \Big  |\sum \limits_{i = 1}^n X_i \Big  |^p}   \Big ] \le R(p) \max \Big  \{\sum \limits_{i = 1}^n \e{\left|X_i\right|^p},  \Big  (\sum \limits_{i = 1}^n \e{X_i^2} \Big  )^{p/2} \Big  \}.
        \end{align*}
    \end{proof}

    \begin{thm} \label{as_equi}
        The process $Z_n$ is asymptotically tight.
    \end{thm}
    \begin{proof}
        We will prove that $Z_n$ satisfies the conditions of Theorem \ref{sup_ineq} resp. Corollary \ref{small_inc} with a function $\eta$ that depends on $n$. 
        Let $1 \le j \le k$ and $s_1, \ldots, s_k \in [0, 1]$ be arbitrary. For $r \le s_j \le t$ define
        \begin{align*}
            R &= (s_1, \ldots, r, \ldots, s_k)^t  , \\
            S &= (s_1, \ldots, s_j, \ldots, s_k)^t ,  \\
            T &= (s_1, \ldots, t, \ldots, s_k)^t , 
        \end{align*}
    then an   application of the Markov- resp. Hölder-inequality yields
        \begin{align}
            &\p{m^j(s_1, \ldots, s_k, r, t) > \lambda} \nonumber\\ 
            \le{}& \p{|Z_n(T) - Z_n(S)| |Z_n(S) - Z_n(R)| > \lambda^2}  \nonumber\\ 
            \le{}& \lambda ^{-2k - 2} \e{|Z_n(T) - Z_n(S)|^{k + 1} |Z_n(S) - Z_n(R)|^{k + 1}}\nonumber\\
            \le{}&  \lambda^{-2k - 2} \left(\e{|Z_n(T) - Z_n(S)|^{2k + 2}} \e{|Z_n(S) - Z_n(R)|^{2k + 2}}\right)^{\frac{1}{2}}. \label{ineq_mi}
        \end{align}
        By condition \ref{cond_moments}, \ref{cond_cov} and Lemma \ref{rosenthal} from Appendix~\ref{sec6} it follows that 
        \begin{align}
            \e{|Z(T) - Z(S)|^{2k + 2}} 
            \le{}& R(2k + 2) \Big  (\sum \limits_{i \in T_n} (g_n(i, T) - g_n(i, S))^2 \e{|X_{n}(i)|^{2k + 2}}^{2/(2k + 2)}\Big  )^{k + 1} \nonumber \\
            \le{}& R(2k + 2) C^{2k + 2}\Big (\sum \limits_{i \in T_n} (g_n(i, T) - g_n(i, S))^2 \e{|X_{n}(i)|^2}\Big  )^{k + 1} \nonumber \\
            ={}& R(2k + 2) C^{2k + 2} \var{Z_n(T) - Z_n(S)}^{k + 1}\nonumber \\
            \le{}& R(2k + 2) C^{3k + 3} \left(t - s_j + h_n^{(j)}\right)^{k + 1} ,   \label{ineq_variance}
        \end{align}
  and  a similar argument yields the inequality
        \begin{align*}
            \e{|Z(S) - Z(R)|^{2k + 2}} \le R(2k + 2) C^{3k + 3} \left(s_j - r + h_n^{(j)}\right)^{k + 1}~.
        \end{align*}
For $t - r \ge h_n^{(j)}$ this implies
        \begin{align*}
            \e{|Z(T) - Z(S)|^{2k + 2}}\e{|Z(S) - Z(R)|^{2k + 2}} 
           \le{}& c_{r,k} \left((t - s_j + h_n^{(j)})(s_j - r + h_n^{(j)}\right)^{k + 1} \\
            \le{}& c_{r,k} \Big (\frac{t - r + 2h_n^{(j)}}{2}\Big )^{k + 1}\\
            \le{}&c_{r,k}  2^{k + 1}\left(t - r\right)^{k + 1}~,
        \end{align*}
        where $c_{r,k} =R(2k + 2)^2 C^{6k + 6} $ 
        This inequality is also correct for $t - r < h_n^{(j)}$, since in this case  $Z_n(T) = Z_n(S)$ or $Z_n(S) = Z_n(R)$ holds by \ref{cond_lattice}. Plugging this inequality into (\ref{ineq_mi}) yields the estimate
        \begin{align}
            \p{m^j(s_1, \ldots, s_k, r, t) > \lambda} \le \lambda^{-2k - 2} R(2k + 2) C^{3k + 3} 2^{k + 1}(t - r)^{k + 1}. \label{ineq_dev}
        \end{align}      
        Let $s, t \in [0, 1]^k$ be arbitrary, set $S_i = (s_1, \ldots, s_{i - 1}, t_i, \ldots, t_k)^t$ and note that
$        Z(t) - Z(s) = \sum \limits_{i = 1}^{k} (Z(S_i) - Z(S_{i + 1})).
 $
         From (\ref{ineq_variance}) we can conclude
        \begin{align}
            &\p{|Z(t) - Z(s)| > \lambda} \nonumber \\
            \le{}& \sum \limits_{i = 1}^k \p{|Z(S_i) - Z(S_{i + 1})| > \lambda/k} \le \sum \limits_{i = 1}^k k^{2k + 2} \lambda^{-2k - 2} \e{|Z(S_i) - Z(S_{i + 1})|^{2k + 2}} \nonumber\\
            \le{}&\sum \limits_{i = 1}^k k^{2k + 2} \lambda^{-2k - 2} R(2k + 2) C^{3k + 3}(|t_i - s_i| + h_n^{(i)})^{k + 1} \nonumber\\
            \le{}& k^{2k + 3} \lambda^{-2k - 2}R(2k + 2) C^{3k + 3}(\|t - s\|_\infty + h_n)^{k + 1}.  \label{ineq_fn}
        \end{align}
        where $h_n = \max \limits_{i = 1}^k h_n^{(i)}$.
        Let $m$ be a positive integer and define for $j \in \{1, \ldots, m\}^k$ the set $K_j = \prod \limits_{i = 1}^k \left[\frac{j_i - 1}{m}, \frac{j_i}{m}\right]$. The inequalities (\ref{ineq_dev}) and (\ref{ineq_fn}) allow us to apply Corollary \ref{small_inc} with $\delta = k + 1$, $\gamma = 2k + 2$ and $\varepsilon = \frac{1}{m}$, which gives  the inequality
        \begin{align*}
            \pr \Big(~{\sup \limits_{s, t \in K_j} |Z(t) - Z(s)| > \lambda} \Big) &\le D(m^{-k - 1}\lambda^{-2k - 2} + \lambda^{-2k - 2}(m^{-1} + h_n)^{k + 1}) \\
            &\le 2D \lambda^{-2k - 2}(m^{-1} + h_n)^{k + 1}
        \end{align*}
  for some constant $D$ that depends only on $k$ and $C$. This yields
        \begin{align*}
            &\limsup \limits_{n \to \infty}\pr \Big ( {\sup \limits_{j \in \{1, \ldots, m\}^k} \sup \limits_{s, t \in K_j} |Z(t) - Z(s)| > \lambda} \Big) \\
            \le{}& \limsup \limits_{n \to \infty} \sum \limits_{j \in \{1, \ldots, m\}^k} \pr \Big( {\sup \limits_{s, t \in K_j} |Z(t) - Z(s)| > \lambda}  \Big ) \\
            \le{}& \limsup \limits_{n \to \infty}  m^k 2D \lambda^{-2k - 2}(m^{-1} + h_n)^{k + 1} = \frac{2D\lambda^{-2k - 2}}{m}\xrightarrow{m \to \infty} 0.
        \end{align*}
        Since the finite-dimensional distributions of $Z_n(t)$ converge weakly by Theorem \ref{fin_conv}, Theorem 1.5.6 in \cite{vanwell1996} yields the asymptotic tightness of $Z_n$.
    \end{proof}
    
    \subsection{Moments of logarithms of beta-distributed random variables} \label{betamom}

   For the application of Theorem \ref{limit_thm}   we need precise estimates of the central moments of the log-beta distribution, which are 
   given in this Section. 
    
    \begin{lemma} \label{cum_mom_ineq}
        Let $n \ge 2$ and $Y$ be a random variable with finite $n$-th moment. Denote by $\mu_n$ resp. $\kappa_n$ the $n$-th central moment resp. the $n$-th cumulant of $Y$. If the inequality
        \begin{align*}
            |\kappa_m| \le C \kappa_2^{m/2}
        \end{align*}
        holds for some constant $C \ge 1$ and all $2 \le m \le n$, then the inequality
        \begin{align*}
            |\mu_m| \le (C+m)^m \mu_2^{m/2}
        \end{align*}
        holds for all $0 \le m \le n$.
    \end{lemma}
    \begin{proof}
        We will show this theorem with an induction argument. For $n = 0, 1, 2$ the inequality holds trivially. For $n \ge 3$ we obtain from 
        the recursion
        \begin{align*}
            \mu_n = \kappa_n + \sum \limits_{m = 2}^{n - 2} \binom{n - 1}{m - 1} \kappa_m \mu_{n - m}
        \end{align*}
     [see for example  \cite{smith1995}] 
             \begin{align*}
            |\mu_m| 
            % \le{}& |\kappa_m| + \sum \limits_{i = 2}^{m - 2} \binom{m - 1}{i - 1} |\kappa_i| |\mu_{m - i}| \\
            \le{}& C\kappa_2^{m/2} + \sum \limits_{i = 2}^{m - 2} \binom{m - 1}{i - 1} C\kappa_2^{i/2}(C + m - i)^{m - i}\kappa_2^{(m - i)/2} \\
            ={}& C\Big \{ 1 + \sum \limits_{i = 2}^{m - 2} \binom{m - 1}{i}(C + i)^i\Big \} \kappa_2^{m/2} 
            % \\  \le{}& 
            \le C v \Big \{   \sum \limits_{i = 0}^{m - 1} \binom{m - 1}{i}(C + m - 1)^i\Big \}  \kappa_2^{m/2} \\
            ={}& C (C + m)^{m - 1} \kappa_2^{n/2} \le (C + m)^m \mu_2^{m/2}. \\
        \end{align*}
    \end{proof}

    \begin{thm} \label{mom_high}
    If  $a, b \ge M > 0$, $X \sim \beta(a, b)$ and $Y = \log(X)$,  then the inequality
        \begin{align*}
            \big| \e{(Y - \e{Y})^n} \big| \le \Big (n! 2^{n/2} (M \wedge 1)^{-(n - 1)/2}\left(1 + \frac{1}{M}\right) + n\Big )^n\var{Y}^{n/2}
        \end{align*}
        holds.
    \end{thm}
    \begin{proof}
        In the following we will show that the cumulants $\kappa_n$ of $Y$ satisfy
        \begin{align} \label{vor}
            |\kappa_m| \le n! 2^{n/2} (M \wedge 1)^{-(n - 1)/2} \Big (1 + \frac{1}{M}\Big ) \kappa_2^{m/2}
        \end{align}
        for all $2 \le m \le n$. An application of Lemma \ref{cum_mom_ineq} then yields the desired result.    
For a proof of \eqref{vor} we denote by  
        \begin{align*}
            K(t) = \log \e{\exp(tY)} = \log \e{X^t} = \log \Big (\frac{B(a + t, b)}{B(a, b)}\Big ) = \log \Big (\frac{\Gamma(a + t) \Gamma(b) \Gamma(a + b)}{\Gamma(a + b + t) \Gamma(a) \Gamma(b)}\Big )
        \end{align*}
                    the cumulant generating function of the random variable $Y$.
        For $m \ge 1$ the $m$-th derivative of $K$ can be calculated as
        \begin{align*}
            K^{(m)}(t) = \psi_{m - 1}(a + t) - \psi_{m - 1}(a + b + t),
        \end{align*}
        where $\psi_k(x) = \frac{d^{k + 1}}{dx^{k + 1}} \log \Gamma(x)$ denotes the polygamma function of order $k$. This yields
        \begin{align}
            \kappa_m = K^{(m)}(0) = \psi_{m - 1}(a) - \psi_{m - 1}(a + b) \label{form_kappa1}
        \end{align}
        for $m \ge 1$. Applying formula  (\ref{bound_psi_diff}) and (\ref{bound_psik_diff}) from Appendix \ref{appendix_polygamma} yields
        \begin{align}
    %        \begin{split}
                \frac{|\kappa_m|}{|\kappa_2|^{m/2}} & \le  {m!  \min(a, b) a^{-m} \left(1 + a^{-1}\right)}{\Big (\frac{a(a+b)}{b}\Big )^{m/2}} 
                 = m! \Big (\frac{1}{a} + \frac{1}{b}\Big )^{m/2} \min(a, b) (1 + a^{-1})\nonumber \\
                 & \le ~m! 2^{m/2} \max(a^{-1}, b^{-1})^{m/2} \max(a^{-1}, b^{-1})^{-1} (1 + a^{-1}) \nonumber \\
                & \le~ m! 2^{m/2} M^{-(m - 1)/2} \Big (1 + \frac{1}{M}\Big ) \le n! 2^{n/2} (M \wedge 1)^{-(n - 1)/2} \Big (1 + \frac{1}{M}\Big ).
                \label{bnd_kappa}
   %         \end{split} 
        \end{align}
    \end{proof}
    
    \begin{thm} \label{mom_high_2}
     If  $a, b \ge M > 0$, $X \sim \beta(a, b)$ and $Y = \log(X(1-X))$, then there exists a constant $C_n(M)$ depending only on $n$ and $M$ such that the inequality
        \begin{align*}
            \big| \e{(Y - \e{Y})^n} \big| \le C_n(M)\var{Y}^{n/2}
        \end{align*}
        holds.
    \end{thm}
    \begin{proof}
        We will show that for $n \ge 2$ the quotient
$
           {|\kappa_n|}/ {\kappa_2^{n/2}}
  $
        is bounded by a constant depending only on $n$ and $M$. The assertion then follows from the same arguments as used in 
        the proof of Theorem \ref{mom_high}. The only difference is that the bound has a more complex structure and we will therefore omit an explicit representation of $C_n(M)$. 
        
        The cumulant-generating function of $Y$ is given by
        \begin{align*}
            K(t) % & = \log \e{e^{tY}} 
            = \log \e{X^t(1-X)^t} 
            = \log \Big (\frac{B(a+t, b+t)}{B(a, b)}\Big) = \log \Big(\frac{\Gamma(a+t)\Gamma(b+t)\Gamma(a+b)}{\Gamma(a+b+2t)\Gamma(a)\Gamma(b)}\Big).
        \end{align*}
        For $n \ge 1$ the $n$-th derivative of the cumulant-generating function can be written as
        \begin{align*}
            K^{(n)}(t) &= \psi_{n - 1}(a+t) + \psi_{n - 1}(b+t) - 2^n \psi_{n - 1}(a + b + 2t),  
        \end{align*}
        and this yields the representation
        \begin{align}     \label{form_kappa2}
                \kappa_n ={}& K^{(n)}(0) = \psi_{n - 1}(a) + \psi_{n - 1}(b) - 2^n \psi_{n - 1}(a+b)\\
                ={}&\Big \{\psi_{n - 1}(a) + \psi_{n - 1}(b) - 2\psi_{n - 1}\Big (\frac{a + b}{2}\Big )\Big\}
                + \Big \{\psi_{n - 1}\Big (\frac{a + b}{2}\Big ) - \psi_{n - 1}\Big (\frac{a + b + 1}{2}\Big )\Big \}
  \nonumber
         \end{align}
        for the $n$-th cumulant of $Y$, where we have used  formula (6.4.8) in \cite{abrsteg1964}. In the following, we will multiple times use the fact that $(-1)^n \psi_{n - 1}$ is a nonnegative decreasing function, which is apparent from formula (\ref{poly_integral}) in Appendix \ref{appendix_polygamma}. 
        
        By the mean-value theorem there exists a $\xi \in \left(\frac{a + b}{2}, \frac{a + b + 1}{2}\right)$ such that the inequality 
        \begin{align} \nonumber 
   \psi_{1}\Big (\frac{a + b}{2}\Big ) - \psi_{1}\Big(\frac{a + b + 1}{2}\Big ) & = -\frac{1}{2} \psi_2(\xi) \ge \frac{1}{2} \Big |\psi_2 \Big (\frac{a + b + 1}{2}\Big )\Big | \\
                & \ge 2 (a + b + 1)^{-2} \ge \frac{1}{\left(1 + \frac{1}{M}\right)^2(a + b)^2}
\label{bnd_k2_small}
        \end{align}
        is satisfied, where the lower bound for $|\psi_2|$ follows from (\ref{bound_psi}) in Appendix \ref{appendix_polygamma}. 
        
        From (\ref{bound_psi}), (\ref{bound_psi_diff}) and (\ref{log_conv}) we know
        \begin{align}
            \begin{split}
                &\psi_1(a) + \psi_1(b) - 2 \psi_1\left(\frac{a + b}{2}\right) \ge \left(\sqrt{\psi_1(a)} - \sqrt{\psi_1(b)}\right)^2 \\
                ={}& \Big (\frac{\psi_1(a) - \psi_1(b)}{\sqrt{\psi_1(a)} +
                 \sqrt{\psi_1(b)}}\Big )^2 \ge \Big (\frac{b - a}{ab ( a^{-1/2} + b^{-1/2} )}\Big)^2\frac{M}{M + 1}\\
                \ge{}&\Big (\frac{b - a}{ab \sqrt{2} (a^{-1} + b^{-1} )^{1/2}}\Big )^2\frac{M}{M + 1} = \frac{(b - a)^2}{2ab(a + b)}\frac{M}{M + 1}.
            \end{split}\label{bnd_k2_large}
        \end{align}
        Now (\ref{form_kappa2}) and (\ref{bnd_k2_large}) yield
        \begin{align*}
            \frac{|\kappa_n|}{\kappa_2^{n/2}} \le \frac{\left|\psi_{n - 1}(a) + \psi_{n - 1}(b) - 2\psi_{n - 1}\left(\frac{a + b}{2}\right) \right|}{\kappa_2^{n/2}} + \frac{\left|\psi_{n - 1}\left(\frac{a + b}{2}\right) - \psi_{n - 1}\left(\frac{a + b + 1}{2}\right)\right|}{\left|\psi_1\left(\frac{a + b}{2}\right) - \psi_1\left(\frac{a + b + 1}{2}\right)\right|^{n/2}}.
        \end{align*}
        The formulas (\ref{form_kappa1}) and (\ref{bnd_kappa}) show that the second term is bounded and it only remains to prove that the first term is bounded. Since the term is zero for $a = b$, we will assume $a \ne b$ in the following. 
        
For this purpose let $0 < c < d$ be positive numbers and set $x = \frac{c + d}{2}$, $h = \frac{d - c}{2}$. Then $x + h = d$ and $x - h = c$. By the generalized mean-value theorem there are numbers $0 < \xi' < \xi < h$ such that
        \begin{align*}
            \frac{\psi_{n - 1}(x + h) + \psi_{n - 1}(x - h) - 2\psi_{n - 1}(x)}{h^2} &= \frac{\psi_n(x + \xi) - \psi_n(x - \xi)}{2\xi} \\
          &=  \frac{\psi_{n + 1}(x + \xi') + \psi_{n + 1}(x - \xi')}{2}
        \end{align*}
        holds. Note that $x + \xi \in \left(\frac{c + d}{2}, d\right)$ and $x - \xi \in \left(c, \frac{c + d}{2}\right)$. Applying this to $c = \min(a, b)$ and $d = \max(a, b)$ and using the monotonicity of $|\psi_{n + 1}|$ yields
        \begin{align}
          \label{bnd_kn}            \Big |\psi_{n - 1}(a) + \psi_{n - 1}(b) - 2\psi_{n - 1}\Big(\frac{a + b}{2}\Big )\Big | &\le \big |\psi_{n + 1}(\min(a, b)) \big| \frac{(b - a)^2}{4} \\
           &      \le{} (n + 1)! \min(a, b)^{-(n + 1)} \Big (1 + \frac{1}{M}\Big ) (b - a)^2, 
      \nonumber 
        \end{align}
        where the upper bound for $|\psi_{n + 1}|$ stems from (\ref{bound_psi}).        
        We will now consider three separate cases.         \br
        
        \textbf{Case 1:} $\frac{1}{2} \le \frac{\min(a, b)}{\max(a, b)}$ and $\frac{(ab(a+b))^{n/2}}{\min(a, b)^{n + 1}} \le |b - a|^{n - 2}$. An application of (\ref{bnd_k2_large}) and (\ref{bnd_kn}) yields the desired result:
        \begin{align*}
    f_n:=         &\frac{\left|\psi_{n - 1}(a) + \psi_{n - 1}(b) - 2 \psi_{n - 1} \left(\frac{a + b}{2}\right)\right|}{\kappa_2^{n/2}} \le (n + 1)! \Big (1 + \frac{1}{M}\Big 
            )\frac{(b - a)^2}{\min(a, b)^{n + 1} \left(\frac{(b - a)^2}{2ab(a+b)}\right)^{n/2}} \\
            ={}& (n + 1)! \Big (1 + \frac{1}{M}\Big )\frac{(2ab(a+b))^{n/2}}{\min(a, b)^{n + 1} |b - a|^{n - 2}} \le (n + 1)! \Big (1 + \frac{1}{M}\Big ) 2^{n/2}.
        \end{align*}

        \textbf{Case 2:} $\frac{1}{2} \le \frac{\min(a, b)}{\max(a, b)}$ and $\frac{(ab(a+b))^{n/2}}{\min(a, b)^{n + 1}} > |b - a|^{n - 2}$. An application of (\ref{bnd_k2_small}) and (\ref{bnd_kn}) yields the desired result:
        \begin{align*}
      %      &\frac{\left|\psi_{n - 1}(a) + \psi_{n - 1}(b) - 2 \psi_{n - 1} \left(\frac{a + b}{2}\right)\right|}{\kappa_2^{n/2}} \\
              f_n    \le{}& (n + 1)! \Big (1 + \frac{1}{M}\Big) \frac{(b - a)^2}{\min(a, b)^{n + 1}} \Big(\Big (1 + \frac{1}{M}\Big ) (a + b)\Big)^n \\
            \le{}& (n + 1)!\Big(1 + \frac{1}{M}\Big)^{n + 1} \Big (\frac{(ab(a+b))^{n/2}}{\min(a, b)^{n + 1}}\Big)^{2/(n-2)} \min(a, b)^{-(n + 1)} (a + b)^n \\
            ={}& (n + 1)!\Big(1 + \frac{1}{M}\Big)^{n + 1} \frac{a^{n/(n - 2)}b^{n/(n - 2)}(a + b)^{n + n/(n - 2)}}{\min(a, b)^{n + 1 + 2(n + 1)/(n - 2)}} \\
            ={}& (n + 1)!\Big(1 + \frac{1}{M}\Big)^{n + 1} \Big(\frac{\max(a, b)}{\min(a, b)}\Big)^{n/(n - 2)} \Big (1 + \frac{\max(a, b)}{\min(a, b)}\Big )^{n + n/(n - 2)} \\
            \le{}& (n + 1)!\Big (1 + \frac{1}{M}\Big )^{n + 1} 3^{n + 2n/(n - 2)}.
        \end{align*}
        
        \textbf{Case 3:} $\frac{\min(a, b)}{\max(a, b)} < \frac{1}{2}$. The estimate (\ref{bound_psi}) yields
        \begin{align*}
            &\Big |\psi_{n - 1}(a) + \psi_{n - 1}(b) - 2 \psi_{n - 1} \Big (\frac{a + b}{2}\Big )\Big | \le |\psi_{n - 1}(a)| + |\psi_{n - 1}(b)| + 2 \Big |\psi_{n - 1} \Big(\frac{a + b}{2}\Big )\Big | \\
            \le{}& 4 |\psi_{n - 1}(\min(a, b))| \le 4 n! \min(a, b)^{-(n - 1)}\Big (1 + \frac{1}{M}\Big ).
        \end{align*}
        Jointly with (\ref{bnd_k2_large}) this implies
        \begin{align*}
   %         &\frac{\left|\psi_{n - 1}(a) + \psi_{n - 1}(b) - 2 \psi_{n - 1} \left(\frac{a + b}{2}\right)\right|}{\kappa_2^{n/2}} \\
         f_n    \le{}& 4 n!\Big(1 + \frac{1}{M}\Big)^{1 + n/2} \min(a, b)^{-(n - 1)} \frac{(2ab(a+b))^{n/2}}{|b - a|^n} \\
            \le{}& 2^{2+n} n!\Big(1 + \frac{1}{M}\Big)^{1 + n/2} \frac{\max(a, b)^n \min(a, b)^{n/2}}{\max(a, b)^n 2^{-n} \min(a, b)^{n - 1}} \\
            ={}& 2^{2+2n} n!\Big(1 + \frac{1}{M}\Big)^{1 + n/2} \min(a, b)^{1 - n/2} \le 2^{2+2n} n!\Big(1 + \frac{1}{M}\Big )^{1 + n/2} M^{1 - n/2} \\
            \le{}& 2^{2+2n} n! \Big(1 + \frac{1}{M}\Big)^n,
        \end{align*}
        where we used the inequality
        \begin{align*}
            |b - a| = \max(a, b) - \min(a, b) > \max(a, b) - \frac{1}{2}\max(a, b) = \frac{\max(a, b)}{2}
        \end{align*}
 in the third line.
    \end{proof}
    
\subsection{Proofs of Lemma~\ref{lemma_cumulant1}  and   Lemma~\ref{lemma_cumulant2}} \label{secc}
    
        {\bf Proof of Lemma \ref{lemma_cumulant1}}
        If $X \sim \beta(a, b)$, then we obtain from formula (\ref{form_kappa1}) in Appendix~\ref{betamom}
        \begin{align*}
            (-1)^m \kappa_m(\log(X)) = (-1)^{m + 1}(\psi_{m - 1}(a + b) - \psi_{m - 1}(a)) = \int_a^{a + b} (-1)^{m + 1} \psi_m(t) \, dt. 
        \end{align*}
As  $(-1)^{m + 1} \psi_m(t) \ge 0$  (see   formula (\ref{poly_integral}) in Appendix \ref{appendix_polygamma})  it follows from  (\ref{bound_psi}) 
that 
        \begin{align}
            (-1)^m \kappa_m(\log(X)) \ge \int_a^{a + b} (m - 1)! t^{-m} \, dt \ge (m - 1)! b (a + b)^{-m}, \label{lower_bound_cumulant}
        \end{align}
        \begin{align}
            (-1)^m \kappa_m(\log(X)) \le \int_a^{a + b} (m - 1)! t^{-m} \left(1 + \frac{m}{t}\right) \, dt \le (m-1)! b a^{-m} \left(1 + \frac{m}{a}\right). \label{upper_bound_cumulant}
        \end{align}
        Applying (\ref{lower_bound_cumulant}) yields the lower bound
        \begin{align*}
            (-1)^{m}\kappa_m(S_n) 
            % &= (-1)^m\sum \limits_{i = 1}^{\lfloor nt \rfloor - 1}\sum \limits_{j = 1}^{\lfloor p_ns \rfloor - 1}(\lfloor nt \rfloor - i)^m \kappa_m(\log(r_{2n + 1, 2i, j})) \\
            &\ge \sum \limits_{i = 1}^{\lfloor nt \rfloor - 1}\sum \limits_{j = 1}^{\lfloor p_ns \rfloor - 1}(\lfloor nt \rfloor - i)^m \frac{j(m - 1)!}{2(p_n + 1)^m(n - i + 1)^m} \\
            &\ge (m - 1)!\frac{\lfloor p_ns  - 1\rfloor\lfloor p_ns\rfloor}{4(p_n + 1)^m} \sum \limits_{i = 1}^{\lfloor nt \rfloor - 1} \frac{(nt - i - 1)^m}{(n - i + 1)^m} \\
            &= (m - 1)!\frac{\lfloor p_ns - 1\rfloor\lfloor p_ns\rfloor}{4(p_n + 1)^m}\sum \limits_{i = 0}^{\lfloor nt \rfloor - 2} \frac{(t - 2/n - i/n)^m}{(1 - i/n)^m} \\
            &\ge n \frac{\lfloor p_ns - 1\rfloor\lfloor p_ns\rfloor}{(p_n + 1)^m} \frac{(m - 1)!}{4} \int_0^{t-2/n} \Big (\frac{t - 2/n - x}{1 - x}\Big )^m \, dx ~,
        \end{align*}
     while the upper bound follows in a similar manner from  (\ref{upper_bound_cumulant}), i.e. 
        \begin{align*}
            & (-1)^{m}\kappa_m(S_n) = (-1)^m\sum \limits_{i = 1}^{\lfloor nt \rfloor - 1}\sum \limits_{j = 1}^{\lfloor p_ns \rfloor - 1}(\lfloor nt \rfloor - i)^m \kappa_m(\log(r_{2n + 1, 2i, j})) \\
            \le{}& \sum \limits_{i = 1}^{\lfloor nt \rfloor - 1}\sum \limits_{j = 1}^{\lfloor p_ns \rfloor - 1}\frac{(\lfloor nt \rfloor - i)^mj(m - 1)!}{2((p_n + 1)(n - i + 1) - j/2)^m}
            \Big ( 1 + \frac{m}{(p_n + 1)(n - i + 1) - j/2}\Big ) \\
            \le{}& \sum \limits_{i = 1}^{\lfloor nt \rfloor - 1}\sum \limits_{j = 1}^{\lfloor p_ns \rfloor - 1}\frac{(nt  - i)^mj(m - 1)!}{2((p_n + 1)^m(n - i)^m}\Big (1 + \frac{m}{(p_n + 1)(n - i)}\Big ) \\
            \le{}&\frac{\lfloor p_ns  - 1 \rfloor \lfloor p_ns\rfloor}{(p_n + 1)^m}\left(1 + \frac{m}{p_n}\right)\frac{(m - 1)!}{4} \sum \limits_{i = 1}^{\lfloor nt \rfloor - 1} \frac{(t - i/n)^m}{(1 - i/n)^m} \\
            \le{}&n\frac{\lfloor p_ns - 1 \rfloor \lfloor p_ns\rfloor}{(p_n + 1)^m}\Big (1 + \frac{m}{p_n}\Big )\frac{(m - 1)!}{4} \int_0^t \Big (\frac{t - x}{1 - x}\Big )^m \, dx.
        \end{align*}
        The bounds for $\kappa_m(S_n')$ are proven in essentially the same way. 

\bigskip

   {\bf Proof of Lemma \ref{lemma_cumulant2}.}
    For a beta distributed random variable  $X \sim \beta(a, b)$ a  similar calculation as given in   (\ref{form_kappa2})  of
    Appendix \ref{betamom}
    shows 
        \begin{align}
   %         \begin{split}
                \kappa_m(\log(X^{d + 1}(1 - X)^d))  \nonumber 
                ={} &(d + 1)^m\psi_{m - 1}(a) + d^m\psi_{m - 1}(b) - (2d + 1)^m\psi_{m - 1}(a + b) \\
                ={}& d^m\big\{\psi_{m - 1}(a) + \psi_{m - 1}(b) - 2^m\psi_{m - 1}(a + b)\big\} \nonumber  \\
                &+ ((d + 1)^m - d^m) \psi_{m - 1}(a) - ((2d + 1)^m - (2d)^m) \psi_{m - 1}(a + b)\nonumber   \\
                ={}& d^m\big\{\psi_{m - 1}(a) + \psi_{m - 1}(b) - 2^m\psi_{m - 1}(a + b)\big\} \nonumber  \\
                &+ \sum \limits_{k = 0}^{m - 2} \binom{m}{k} d^k(\psi_{m - 1}(a) - 2^k \psi_{m - 1}(a + b)) \nonumber  \\
                &+ m d^{m - 1}(\psi_{m - 1}(a) - 2^{m - 1}\psi_{m - 1}(a + b)).            \label{cumulant_special}
   %         \end{split}
        \end{align}
        Using the same calculations as in (\ref{form_kappa2}) we get
        \begin{align*}
            \psi_{m - 1}(a) + \psi_{m - 1}(b) - 2^m\psi_{m - 1}(a + b)
            ={} &\Big \{\psi_{m - 1}(a) + \psi_{m - 1}(b) - 2\psi_{m - 1}\Big (\frac{a + b}{2}\Big )\Big \}\\
                &+ \Big \{\psi_{m - 1}\Big (\frac{a + b}{2}\Big) - \psi_{m - 1}\Big (\frac{a + b + 1}{2}\Big )\Big\}
        \end{align*}
        Applying (\ref{bnd_kn}), the mean-value theorem and the estimate (\ref{bound_psi}) from Appendix~\ref{appendix_polygamma} yields
        \begin{align}
  \nonumber 
   |\psi_{m - 1}(a) + \psi_{m - 1}(b) - 2^m\psi_{m - 1}(a + b)| 
                \le{}&  \Big (1 + \frac{1}{\min(a, b)}\Big ) \frac{(m + 1)!  (b - a)^2}{\min(a, b)^{(m+1)}} \\
                &+ m! 2^m(a  + b)^{-m} \Big (1 + \frac{2}{a + b}\Big  ).    \nonumber  \\
                \le{}& \frac{ (m + 1)! 2^m}{\min(a, b)^{m} }\Big (1 + \frac{2}{\min(a, b)}\Big) \Big (\frac{(b - a)^2}{\min(a, b)} + 1\Big). 
          \label{cumulant_bound1}
        \end{align}
        The second part of the sum can be approximated using (\ref{bound_psi})
        \begin{align}
            \label{cumulant_bound2}
            \begin{split}
                \Big  |\sum \limits_{k = 0}^{m - 2} \binom{m}{k} d^k(\psi_{m - 1}(a) - 2^k \psi_{m - 1}(a + b))\Big | 
                \le{}& 2|\psi_{m - 1}(a)| \sum \limits_{k = 0}^{m - 2} \binom{m}{k} (2d)^k \\
                \le{}& 2|\psi_{m - 1}(a)| (2d + 1)^{m - 2} \\
                \le{}& (m - 1)! 2^{m - 1} \frac{(d + 1)^{m - 2}}{a^{m - 1}} \Big (1 + \frac{1}{a}\Big  ). 
            \end{split}
        \end{align}
        For the last summand, we again use the same formula as in (\ref{form_kappa2}), together with the bound (\ref{bnd_kappa})
        \begin{align}
            \label{cumulant_bound3}
            \begin{split}
                m d^{m - 1}|\psi_{m - 1}(a) - 2^{m - 1}\psi_{m - 1}(a + b)| 
                \le{}& m d^{m - 1}\Big | \psi_{m - 1}(a) - \psi_{m - 1}\Big  (\frac{a + b}{2}\Big )\Big  | \\
                &+ \frac{m d^{m - 1}}{2}\Big  |\psi_{m - 1}\Big (\frac{a + b}{2}\Big  ) - \psi_{m - 1}\Big  (\frac{a + b + 1 }{2}\Big )\Big | \\
                \le{}& m d^{m - 1} \frac{|b - a| + 1}{2} |\psi_m(\min(a, b))| \\
                \le{}& (m + 1)! \frac{|b - a| + 1}{2} \frac{d^{m - 1}}{\min(a, b)^m} \Big  (1 + \frac{1}{\min(a, b)}\Big ). 
            \end{split}
        \end{align}
      
        Finally, note that
 $
            \min\big\{\frac{p_n + 1}{2}(2n - 2i + 2) - \frac{j}{2}, \frac{p_n + 1}{2}(2n - 2i + 2)\big\} 
         %   = (p_n + 1)(n - i + 1/2) + \frac{p_n + 1 - j}{2}  
         \ge 1.
$
        Combining this inequality and plugging the inequalities (\ref{cumulant_bound1}), (\ref{cumulant_bound2}) and (\ref{cumulant_bound3}) into (\ref{cumulant_special}) shows the desired result
        \begin{align*}
            |\kappa_m(T_n)| 
            \le{}& \sum \limits_{i = 1}^{\lfloor nt \rfloor}\sum \limits_{j = 0}^{\lfloor p_ns \rfloor - 1} \Big\{ \frac{3 (m + 1)! 2^m (\lfloor nt \rfloor - i)^m}{((p_n + 1)(n - i + 1/2))^m}   \cdot 
            \Big (\frac{j^2}{4(p_n + 1) (n - i + 1/2)} + 1   \Big  ) \\
            &\qquad + (m - 1)! 2^{m } \frac{(\lfloor nt \rfloor - i + 1)^{m - 2}}{((p_n + 1 )(n - i + 1/2))^{m - 1}} \\
            &\qquad + (m + 1)! ({j + 1}) \frac{({j + 1}) (\lfloor nt \rfloor - i)^{m - 1}}{((p_n + 1)(n - i + 1/2))^m}  \Big\} \\
            \le{}& 3 (m + 1)!\sum \limits_{i = 1}^{n}\sum \limits_{j = 0}^{p_n - 1} \Big\{ \frac{2^m}{p_n^m}\Big (\frac{j^2}{(p_n + 1) (2n - 2i + 1)} + 1\Big ) +   \frac{p_n^{1 - m} 4^m}{n - i + 1/2}  \Big\} \\
            \le{}& 6 \cdot 4^m (m + 1)!\sum \limits_{i = 1}^{n}\sum \limits_{j = 0}^{p_n - 1} \Big\{p_n^{1 - m}\frac{1}{n - i + 1} + p_n^{-m} \Big\} \\
            \le{}& 6 \cdot 4^m (m + 1)!p_n^{-m}(n p_n + (\log(n) + 1)p_n^2) 
            \le 12 \cdot 4^m (m + 1)!n p_n^{2-m}.
        \end{align*}
        The inequality for $\kappa_m(T_n')$ is proven in essentially the same way. Since the exponents of $p_{2n + 1, 2i - 1, j}$ and $1 - p_{2n + 1, 2i - 1, j}$ in the definition of $T_n'$ are the same, we do not  need (\ref{cumulant_special}) and can apply (\ref{cumulant_bound1}) directly to the formula (\ref{form_kappa2}) of the cumulant. This makes the proof for $T_n'$ much easier and the details are omitted for the sake of brevity.

    \subsection{On the Polygamma functions} \label{appendix_polygamma}
    Throughout this section $a, b, t$ and $z$ will be positive real numbers. \br

    Recall the definition of  
   the Polygamma function in \eqref{polgma} and note that formula $6.4.1$ from \cite{abrsteg1964} states
    \begin{align}
        \psi_k(z) = (-1)^{k + 1} \int_0^\infty \frac{t^ke^{-tz}}{1 - e^{-t}} dt. \label{poly_integral}
    \end{align}
   In particular, this implies that $\psi_k$ is positive and decreasing for odd $k$ and negative and increasing for even $k$.  
    Observing the estimate 
       \begin{align*}
        |\psi_n(z)| = \int_0^\infty \frac{t^ne^{-zt}}{1 - e^{-t}}dt \ge \int_0^\infty t^{n - 1} e^{-zt} dt = \int_0^\infty s^{n - 1} e^{-s} z^{-n} ds = (n - 1)! z^{-n}
    \end{align*}
    we obtain the inequality
    \begin{align}
        (n - 1)! z^{-n} \le |\psi_n(z)| \le (n - 1)! z^{-n} + n! z^{-(n + 1)} \le n!z^{-n} \Big (1 + \frac{1}{z}\Big ) \label{bound_psi},
    \end{align}
    where we have used the inequality $\frac{1}{1 - e^{-t}} \le 1 + \frac{1}{t}$ for the upper bound. From  (\ref{bound_psi}) we obtain the inequalities
    \begin{align}
        \frac{b}{a(a + b)} \le \psi_1(a) - \psi_1(a + b)         = \int \limits_a^{a + b} |\psi_2(x)| dx \le \Big(1 + \frac{2}{a}\Big ) \frac{b}{a(a + b)}
         \label{bound_psi_diff}.
    \end{align}  
    This method also yields upper bounds for differences of higher-order polygamma functions. Note that $|\psi_k|$ is decreasing, which yields together with (\ref{bound_psi})
    \begin{align*}
        |\psi_k(a) - \psi_k(a + b)| = \int_a^{a + b} |\psi_{k + 1}(t)| dt \le b |\psi_{k + 1}(a)| \le b (k + 1)! a^{-(k + 1)}\Big (1 + \frac{1}{a}\Big ).
    \end{align*}
    From the monotonicity of $\psi_k$ and (\ref{bound_psi}) we can furthermore deduce
    \begin{align*}
        |\psi_k(a) - \psi_k(a + b)| \le |\psi_k(a)| \le k! a^{-k} \Big(1 + \frac{1}{a}\Big).
    \end{align*}
    Combining these inequalities proves
    \begin{align}
        |\psi_k(a) - \psi_k(a + b)| \le (k + 1)! \min(a, b) a^{-(k + 1)} \Big (1 + \frac{1}{a}\Big). \label{bound_psik_diff}
    \end{align}
   Inequality (\ref{bound_psi}) also yields
    \begin{align}
        &\Big |\psi_1(a) - \psi_1(a + b) - \tfrac{b}{a(a + b)}\Big |  
        = \Big |\psi_1(a) - \psi_1(a + b)+ \left(\tfrac{1}{a} - \tfrac{1}{a + b}\right)\Big | \le \frac{2}{a^2} + \frac{2}{(a + b)^2} \le \frac{4}{a^2} \label{var_first_order}
    \end{align}
    and
    \begin{align*}
        &\Big |\psi_1(a) + \psi_1(b) - 4 \psi_1(a + b) - \frac{(a - b)^2}{ab(a + b)}\Big |  \\
   %     ={}& \Big |\psi_1(a) + \psi_1(b) - 4 \psi_1(a + b) - \Big (\frac{1}{a} + \frac{1}{b} - \frac{4}{a + b}\Big )\Big | \\
        \le{}& \Big |\psi_1(a) - \tfrac{1}{a} \Big | + \Big |\psi_1(b) - \tfrac{1}{b} \Big | + 4\Big |\psi_1(a + b) - \tfrac{1}{a + b} \Big | \\
        \le{}& \frac{2}{a^2} + \frac{2}{b^2} + \frac{8}{(a + b)^2} \le \frac{6}{(a \wedge b)^2},
    \end{align*}
    where the latter implies
    \begin{align}
        \left|\psi_1(a) + \psi_1(b) - 4 \psi_1(a + b) \right| \le \Big (6 + \frac{(a - b)^2}{(a \wedge b)} \Big  )\frac{1}{(a \wedge b)^2}. \label{bnd_var}
    \end{align}
    
    Finally, $|\psi_n|$ is log-convex by formula (1.4) from \cite{alzer2001}, i.e.
    \begin{align}
        \Big  |\psi_n\Big   (\frac{a + b}{2}\Big  )\Big  | \le \sqrt{|\psi_n(a) \psi_n(b)|} \le \frac{|\psi_n(a)| + |\psi_n(b)|}{2}. \label{log_conv}
    \end{align}

\end{document}